\documentclass[11pt]{amsart}
\usepackage{tikz-cd}
\usepackage{amsmath}
\usepackage{amssymb}
\usepackage{xcolor}
\usepackage{geometry}
\usepackage{enumerate}
\usepackage[alphabetic]{amsrefs}
\usepackage{hyperref}

\DeclareMathOperator{\ric}{Ric}
\DeclareMathOperator{\scal}{scal}
\DeclareMathOperator{\tr}{tr}
\DeclareMathOperator{\hess}{Hess}
\newcommand{\Aut}{\text{Aut}}
\DeclareMathOperator{\GL}{GL}
\newcommand{\Ad}{\text{Ad}}
\DeclareMathOperator{\Span}{Span}
\DeclareMathOperator{\ad}{ad}
\DeclareMathOperator{\Diag}{Diag}
\DeclareMathOperator{\vol}{vol}
\DeclareMathOperator{\Diff}{Diff}
\DeclareMathOperator{\Div}{div}
\DeclareMathOperator{\Iso}{Iso}

\DeclareMathOperator{\Sp}{Sp}
\DeclareMathOperator{\SU}{SU}

\DefineSimpleKey{bib}{primaryclass}{}
\DefineSimpleKey{bib}{archiveprefix}{}

\BibSpec{arXiv}{%
  +{}{\PrintAuthors}{author}
  +{,}{ \textit}{title}
  +{}{ \parenthesize}{date}
  +{,}{ arXiv }{eprint}
  +{,}{ primary class }{primaryclass}
}

\newtheorem*{theorem*}{Theorem}
\newtheorem{theorem}{Theorem}[section]
\newtheorem{proposition}[theorem]{Proposition}
\newtheorem{corollary}[theorem]{Corollary}
\newtheorem{lemma}[theorem]{Lemma}
\theoremstyle{definition}
\newtheorem{definition}[theorem]{Definition}

\theoremstyle{remark}
\newtheorem{remark}[theorem]{Remark}

\numberwithin{equation}{section}

\begin{document}

\title[On Homogeneous Closed Gradient Laplacian Solitons]{On Homogeneous Closed Gradient Laplacian Solitons}

\author{}
\address{}
\curraddr{}
\email{}
\thanks{}

\author{Nicholas Ng}
\address{215 Carnegie Building\\
 Dept. of Mathematics, Syracuse University\\
 Syracuse, NY, 13244.}
\email{nng102@syr.edu}
\urladdr{\url{https://sites.google.com/view/nicholas-ng-math}}
\thanks{The author was partially supported by grant NSF-\#1654034}

\subjclass[2020]{53C24}
\keywords{almost abelian solvmanifold, closed $G_2$-structure, homogeneous, gradient Laplacian soliton, Laplacian flow, nilpotent, structure theorem}

\date{}

\dedicatory{}

\begin{abstract} 
We prove a structure theorem for homogeneous closed gradient Laplacian solitons and use it to show some examples of closed Laplacian solitons cannot be made gradient. More specifically, we show that the Laplacian solitons on nilpotent Lie groups found by Nicolini in \cite{Nic18} are not gradient up to homothetic $G_2$-structures except for $N_1$, where $f$ must be a Gaussian. We also show that the closed $G_2$-structure $\varphi_{12}$ on $N_{12}$ constructed in \cite{FFM16} cannot be a gradient soliton. We further show that closed non-torsion-free gradient Laplacian solitons on almost abelian solvmanifolds are isometric to products $N \times \mathbb R^k$ with $f$ constant on $N$.  
\end{abstract}

\maketitle

\section{Introduction}

\subsection{$G_2$-structures and the Laplacian flow}

Let $M$ be a differentiable $7$-manifold. A $3$-form $\varphi \in \Omega^3(M)$ is a \textit{$G_2$-structure} if at each point $p \in M$, $\varphi_p$ is \textit{positive}, i.e., there exists a basis $(e_i)_{i = 1}^7$ of $T_pM$ such that $$\varphi_p = e^{127} + e^{347} + e^{567} + e^{135} - e^{146} - e^{236} - e^{245},$$ where $e^{ijk} = e^i \wedge e^j \wedge e^k$ and $(e^i)_i$ is the dual basis to $(e_i)_i$. Any such $3$-form $\varphi$ induces a Riemannian metric and an orientation via $$g_\varphi(X,Y)\vol_\varphi = \frac{1}{6}\iota_X(\varphi)\wedge\iota_Y(\varphi)\wedge \varphi,$$ which in turn determines a \textit{Hodge star operator} $*_\varphi:\Omega^k(M) \rightarrow \Omega^{7 - k}(M)$. It is a well known fact that a smooth $7$-manifold admits a $G_2$-structure if and only if it is orientable and spin (see [Proposition 4.18, \cite{Kar09}]).

The unique \textit{torsion forms} $\tau_i \in \Omega^i(M)$, $i = 0, 1,2 ,3$ of a $G_2$-structure $\varphi$ are the components of the \textit{intrinsic torsion} $\nabla \varphi$. \textit{Torsion-free} (or \textit{parallel}) $G_2$-structures, i.e., $\varphi$ such that $\nabla \varphi = 0$ (or equivalently $d\varphi = 0$ and $d*\varphi = 0$; $\tau_i = 0$ for all $i = 0, 1, 2, 3$), have the property that its induced metric $g_\varphi$ has holonomy in $G_2$. Existence of such metrics were suggested by Berger's classification theorem in 1955 and first examples were constructed by Bryant-Salamon in the 1980s. Metrics with holonomy in $G_2$ are hard to find, yet are desirable as they are necessarily Ricci-flat and are important in string theory. We note that many examples of such metrics have since been constructed (see, e.g., references in section 6.2, \cite{Kar20}).

\textit{Closed} (or \textit{calibrated}) $G_2$-structures have the property that $d \varphi = 0$. This class of $G_2$-structures are of interest as they are ``close'' to being torsion-free in the sense that  $\tau_2$, which we write as $\tau_\varphi$, is the only surviving torsion form. Bryant in \cite{Bry06} introduced a natural geometric flow of $G_2$-structures called the \textit{Laplacian flow} given by
$$\begin{cases}
\partial_t\varphi(t) = \Delta_{\varphi(t)}\varphi(t) \\
\varphi(0) = \varphi
\end{cases},$$ 
where $\varphi$ is the initial $3$-form and  $\Delta_{\varphi(t)} =  *_\varphi d *_\varphi d  - d *_\varphi d *_\varphi$ is the  \textit{Hodge Laplacian operator} on $3$-forms. By studying how closed $G_2$-structures deform to torsion-free $G_2$-structures under the flow, Bryant and his collaborators thought that it may help in constructing metrics with holonomy in $G_2$.

It is known that a Laplacian flow solution $\varphi(t)$ starting at $\varphi$ is \textit{self-similar}, i.e., $\varphi(t) = c(t)f(t)^*\varphi$ where $c(t) \in \mathbb R^*$ and $f(t) \in \mathfrak \Diff(M)$,  if and only if \begin{equation}\label{eq:1.1} \Delta_{\varphi}\varphi = \lambda\varphi + \mathcal L_X\varphi, \, \, \, \, \, \, \,  \, \, \, \lambda \in  \mathbb R, \, \, X \in \mathfrak X(M) \, \, \text{complete}.\end{equation} A triple $(\varphi, X, \lambda)$ satisfying (\ref{eq:1.1}) is called a \textit{Laplacian soliton}. The Laplacian soliton is said to be \textit{steady, shrinking, or expanding} if $\lambda = 0$, $\lambda < 0$, or $\lambda > 0$, respectively. A Laplacian soliton is \textit{gradient} if $X$ is a gradient field, i.e., $X = \nabla f$ for some smooth function $f:M \rightarrow \mathbb R$. In the last decade, Laplacian solitons on homogeneous spaces have received increased interest and many new examples have been found (see, e.g., \cites{FR20, Lau17a, Lau17b, LN20, Nic18}, and \cite{Nic22}).

Given a Laplacian flow, there is a corresponding flow equation of the induced metric $g_\varphi$, which can be interpreted as a perturbation of the Ricci flow by quadratic terms in $\tau_\varphi$. Lotay-Wei in \cite{LW17} showed that via an injective linear map, one obtains an equation in terms of the metric from the Laplacian soliton equation of the form \begin{equation}\label{eq:1.2}\frac{1}{2}\mathcal L_Xg_\varphi = -q_\varphi - \frac{1}{3}\lambda g_\varphi, \, \, \, \, \, \lambda \in \mathbb R, \, \, X \in \mathfrak X(M),\end{equation} where $q_\varphi$ is a symmetric $2$-tensor involving the (Ricci) curvature and the square of the skew-symmetric torsion form $\tau_\varphi$ for closed $G_2$-structures, $\tau_\varphi^2$. We also call (\ref{eq:1.2}) the Laplacian soliton equation (in terms of the induced metric $g_\varphi$). When $X = \nabla f$ is a gradient field, (\ref{eq:1.2}) is the \textit{gradient Laplacian soliton equation} and we call the triple $(\varphi, \nabla f, \lambda)$ satisfying (\ref{eq:1.2}) a \textit{gradient Laplacian soliton}. An open question is whether there exists compact Laplacian solitons that are gradient. In fact, the existence of non-trivial Laplacian solitons on compact $7$-manifolds is still open. Motivated by lines of inquiry suggested by Lauret and others, we investigate \textit{closed gradient Laplacian solitons} on homogeneous spaces.

\subsection{Summary of Main Results}

By applying a structure theorem of Petersen-Wylie [Theorem 3.6, \cite{PW22}] in the setting of gradient Laplacian solitons, we obtain the following structure theorem for homogeneous closed gradient Laplacian solitons.

\begin{theorem}[Structure Theorem] Let $(M, \varphi)$ be a $7$-dimensional homogeneous space admitting a closed gradient Laplacian soliton $(\varphi, \nabla f, \lambda)$ where $f$ non-constant.

\begin{enumerate}
    \item The square of the intrinsic torsion $\tau_\varphi^2$ is divergence-free if and only if $(M, g_\varphi)$ is isometric to a product $N \times \mathbb R^k$ where $f$ is constant on $N$.
    
    \item If the square of the intrinsic torsion $\tau_\varphi^2$ is not divergence-free, then either 
    
    \begin{enumerate}
    
        \item $(M, g_\varphi)$ is a one-dimensional extension; $g_\varphi = dr^2 + g_r$; and $f(x, y) = ar + b$; or

        \item $(M, g_\varphi)$ is isometric to a product $N \times \mathbb R^k$ where $N$ is a one-dimensional extension and $f(x, y) = ar(x) + v(y)$, where $v$ is a function on $\mathbb R^k$ and $r$ is a distance function on $N$. 
    \end{enumerate}
\end{enumerate}

\end{theorem}

\begin{remark} Podest\`a-Raffero showed that there are no compact homogeneous $7$-manifolds with $G$-invariant closed $G_2$-structure in \cite{PR19}. Hence there are no compact homogeneous closed gradient Laplacian solitons. One can also see this from the Structure Theorem as each structure has a factor of $\mathbb R^k$.
\end{remark}

\begin{remark} When $\tau_\varphi^2$ is divergence-free, the Structure Theorem implies that if $(\varphi, \nabla f, \lambda)$ is a gradient Laplacian soliton, then the restrictions of soliton equation $$q_{\varphi}\big|_N = -\frac{1}{3}\lambda g_\varphi\big|_N \, \, \, \, \, \text{and} \, \, \, \, \, \hess f\big|_{\mathbb R^k} + q_\varphi\big|_{\mathbb R^k} = -\frac{1}{3}\lambda g_\varphi\big|_{\mathbb R^k}$$ must hold simultaneously. We show these equations cannot hold simultaneously in the divergence-free cases on some of the nilpotent Lie groups studied in Section~\ref{sec:3}. This suggests that we should consider one-dimensional extensions for examples of gradient solitons. Manero examined closed $G_2$-structures on one-dimensional extensions with Lie algebra $\mathfrak g = \mathfrak h \oplus_D \mathbb R e_7$ where the derivation $D$ is the real represenation of some $A \in \mathfrak {sl}(3, \mathbb C)$ in \cite{Man15}. Manero showed such closed $G_2$-structures with $D$ of this form is equivalent to $\mathfrak h$ admitting a symplectic half-flat $\SU(3)$-structure. We study gradient Laplacian solitons on one-dimensional extensions in Section~\ref{sec:4}.
\end{remark}

We prove the Structure Theorem along with some preliminary results in Section~\ref{sec:2}. We define the notion of an ``orthogonally nice basis'' and make some observations concerning the computation of $\Div \tau_\varphi^2$. We also obtain a ``Key Lemma'':  $g_\varphi(\ric_\varphi(\nabla f), \cdot) = -2^{-1}\Div\tau_{\varphi}^2(\cdot)$ on homogeneous spaces. Using observations regarding orthogonally nice bases, the notion of a $G_2$-structure being \textit{Laplacian flow diagonal} (introduced by Lauret in \cite{Lau17a}), and the Key Lemma, we obtain a corollary of the Structure Theorem.

\begin{corollary}
Let $G$ be a Lie group with closed $G_2$-structure $\varphi$ that is Laplacian flow diagonal with respect to an orthogonally nice orthonormal basis $(e_i)_i$. Suppose $\ric_{\varphi}$ is diagonal with respect to $(e_i)_i$. If $(\varphi, \nabla f, \lambda)$ is a gradient Laplacian soliton, then $G$ must be a product metric $\mathbb R^k \times N$ with $f$ constant on $N$. If in addition the kernel of the Ricci tensor is trivial, then $\varphi$ is not a gradient soliton.
\end{corollary}

\noindent The Key Lemma is also used to obtain the following statement.

\begin{proposition}
 If $(\varphi, \nabla f, \lambda)$ is a homogeneous closed gradient Laplacian soliton and $\tau_\varphi^2$ is divergence-free, then  $2^{-1}D_X\|\nabla f\|^2 = 3^{-1}(\scal_\varphi - \lambda)g(\nabla f, X)$ for all $X \in TM.$ If in addition $\|\nabla f\| = $ constant, then $\lambda = \scal_\varphi$. Since $\scal_\varphi \leq 0$ for closed $G_2$-structures, the soliton is either shrinking or steady.
\end{proposition}

The main observation of this paper is that given a homogeneous closed gradient Laplacian soliton, whether $\tau_\varphi^2$ is divergence-free or not determines the structure of the space via the Structure Theorem. If the possible structures cannot occur (e.g., via a contradiction), then the soliton cannot be gradient. Thus we can use the Structure Theorem to ``eliminate'' homogeneous closed gradient Laplacian solitons, i.e., we can use the Structure Theorem to show some examples of closed Laplacian solitons or closed $G_2$-structures cannot be made gradient. This idea is used in Section~\ref{sec:3} where we ``eliminate'' gradient Laplacian solitons on some nilpotent Lie groups.

To study existence of gradient Laplacian solitons on homogeneous manifolds, we ask whether the homogeneous closed Laplacian solitons that have already been found are gradient. Conti-Fern\'{a}ndez showed in \cite{CF11} that there are twelve isomorphism classes, $\{\mathfrak n_i\}_{i = 1}^{12}$, of $7$-dimensional nilpotent Lie algebras admitting closed $G_2$-structures. Fern\'{a}ndez-Fino-Manero in \cite{FFM16} studied which of these nilpotent Lie algebras admit Ricci soliton associated metrics. Whether these nilpotent Lie algebras admit Laplacian solitons is investigated by Nicolini in \cite{Nic18}. Nicolini found Laplacian solitons $(\mathfrak n_i, \varphi_i)$ on the first seven of the twelve nilpotent Lie algebras. In short, Nicolini exhibited the unique symmetric operator $Q_{\varphi_i} \in \mathfrak {gl}(\mathfrak n_i)$ obtained from the closed $G_2$-structures $\varphi_i$'s she constructed as a characterization of either an algebraic or a semi-algebraic soliton. In Section~\ref{sec:3}, we determine whether the Laplacian solitons found by Nicolini are gradient. We show that only the trivial $7$-dimensional nilpotent Lie group admits a gradient Laplacian soliton and that it must be a Gaussian.

\begin{theorem}
The closed Laplacian solitons $\varphi_i$ on $N_i$ for $i = 2,3,4,5,6,7$ found by Nicolini in \cite{Nic18} are not gradient up to homethetic $G_2$-structures. If $N_1$ does admit a gradient Laplacian soliton, then it must be a Gaussian. 
\end{theorem}

\noindent We also obtain the following.

\begin{proposition}
The closed $G_2$-structure $\varphi_{12}$ on $N_{12}$ as constructed in \cite{FFM16} is not gradient up to homethetic $G_2$-structures. 
\end{proposition}

In Section~\ref{sec:4}, we use the Structure Theorem to obtain the structure of almost abelian solvmanifolds admitting closed gradient Laplacian solitons. We note that Lauret exhibits non-gradient examples in this case (see \cite{Lau17a}).

\begin{theorem}
If $(\varphi, \nabla f, \lambda)$ is a closed non-torsion-free gradient Laplacian soliton on Lie group $G_D$ with Lie algebra $\mathfrak g = \mathfrak h \oplus_D \mathbb Re_7$ and $\mathfrak h$ is a codimension-one abelian ideal, then it must be a product $N \times \mathbb R^k$ and $f$ is constant on $N$.
\end{theorem}

\section{Preliminaries}\label{sec:2}

\subsection{Background} For a $7$-dimensional real vector space $\mathfrak p$, a \textit{fixed positive $3$-form} $\varphi_0 \in \Lambda^3\mathfrak p^*$ is given by $$\varphi_0 = e^{127} + e^{347} + e^{567} + e^{135} - e^{146} - e^{236} - e^{245},$$ where $(e^i)_i$ is the basis for $\mathfrak p^*$ dual to basis $(e_i)_i$ for $\mathfrak p$. A $3$-form $\psi \in \Lambda^3 \mathfrak p^*$ is \textit{positive} if there is an $h \in \GL(\mathfrak p)$ such that $h \cdot \varphi_0 = \psi$, i.e., $\psi$ is positive if it is in the $\GL(\mathfrak p)$-orbit of $\varphi_0$. Any positive $3$-form $\psi$ induces a unique inner product and orientation via $\left<X, Y\right>_{\psi}\vol_\psi = 6^{-1}\iota_X(\psi) \wedge \iota_Y(\psi) \wedge \psi$, which together determines a unique Hodge star operator.

Given a smooth $7$-manifold $M$, by identifying $T_pM$ with $\mathfrak p \cong \mathbb R^7$, we say that $\varphi$ is a closed $G_2$-structure if $\varphi_p$ is positive for each $p \in M$ and $d\varphi = 0$. As noted in the introduction, $\varphi$ induces a Riemannian metric $g_\varphi$ and an orientation via $\vol_\varphi$, which together induces a Hodge star operator. It is a well known fact that for closed $G_2$-structures $(M, \varphi)$, $\scal_\varphi \leq 0$ and is equal to $0$ if and only if $\varphi$ is torsion-free. Also, recall that for closed $G_2$-structures, the torsion $2$-form $\tau_\varphi$ satisfies \begin{equation}\label{eq:2.1}
 \tau_{\varphi} = - *_\varphi d *_\varphi \varphi \, \, \, \, \, \text{and} \, \, \, \, \, \Delta_{\varphi}\varphi = d\tau_{\varphi} = - d *_\varphi d *_\varphi \varphi.
\end{equation} Observe that $\tau_\varphi$ is determined by the closed $G_2$-structure $\varphi$ along with the structure equations as it involves the exterior derivative $d$. Both of these are in turn obtained with respect to an orthonormal basis $(e_i)_i$ for $T_pM$ and its dual basis $(e^i)_i$ for $T_p^*M$.

When $(M, \varphi)$ is homogeneous, $M$ has a presentation $M = G/K$ for some Lie subgroup $G \subset \Aut(M, \varphi) \subset \Iso(M, \varphi)$ and isotropy subgroup $K \subset G$. In this setting, we have the following.
\begin{enumerate}[(i)]
    \item $\varphi$ is a $G$-invariant $G_2$-structure on $M$.
    \item When $\mathfrak g = \mathfrak k \oplus \mathfrak p$ is a \textit{reductive deomposition}, i.e., $\Ad(K)\mathfrak p \subset \mathfrak p$, we can identify $\mathfrak p$ with $T_pM$ such that any $G$-invariant $G_2$-structure $\varphi$ is determined by a fixed positive $3$-form on $\mathfrak p$.
\end{enumerate}

\begin{remark} There is a one-to-one correspondence between left-invariant $G_2$-structures on simply-connected Lie groups and $G_2$-structures on its associated Lie algebra (see, e.g., \cite{Fre13} or \cite{Lau17a}). Thus it is common in the literature to identify Lie groups $G = G_\mu$ admitting a (closed) left-invariant $G_2$-structure (similarly, Laplacian, algebraic, or semi-algebraic soliton) $\varphi$ with its corresponding Lie algebra $(\mathfrak g, \mu = [\cdot, \cdot])$ admitting the fixed positive $3$-form $\varphi$. We refer to both $(G, \varphi)$ and $(\mathfrak g, \varphi)$ as $G_2$-structures and further, as Laplacian solitons if $\varphi$ satisfies (\ref{eq:1.1})  for some $\lambda \in \mathbb R$.\end{remark}

In the notation of \cite{LW17}, the associated metric $g_\varphi$ for a closed Laplacian soliton $\varphi$ satisfies \begin{equation}\label{eq:2.2}\begin{cases}\partial_t g_{\varphi}(t) = 2h(t) \\ g_\varphi(0) = g_\varphi\end{cases},\end{equation} where the $2$-form $h$ in local coordinates is $$h_{ij} = -R_{ij} - \dfrac{1}{3}|T|^2g_{ij} - 2T_i^kT_{kj};$$ $R_{ij}$ is the Ricci tensor; $T$ is the full torsion tensor for closed $G_2$-structures. Note that the flow (\ref{eq:2.2}) is a perturbation of the Ricci flow by ``quadratic'' torsion terms. Lotay-Wei uses (\ref{eq:1.1}) and injectivity of the linear map $i_\varphi:S^2(T^*M) \rightarrow \Lambda^3(T^*M)$ to show that the associated metric must also satisfy $$-R_{ij} - \dfrac{1}{3}|T|^2g_{ij} - 2T_i^kT_{kj} = \dfrac{1}{3}\lambda g_{ij} + \dfrac{1}{2}(\mathcal  L_X g)_{ij},$$ or equivalently \begin{equation}\label{eq:2.3} h - \dfrac{1}{3}\lambda g -\dfrac{1}{2}\mathcal L_X g = 0.\end{equation} We refer the reader to [Proposition 9.4, \cite{LW17}] or [Corollary 3.2, \cite{Kar09}] for more details.

We are primarily interested in (\ref{eq:2.3}) cast in the notation of Lauret's papers \cites{Lau17a, Lau17b}. For any closed $G_2$-structure $\varphi$, [Proposition 2.2, \cite{Lau17a}] states that there is a unique symmetric operator \begin{equation}\label{eq:2.4} Q_\varphi = \ric_{\varphi} - \dfrac{1}{12}\tr(\tau_{\varphi}^2)I + \dfrac{1}{2}\tau_{\varphi}^2,\end{equation} where  $\ric_{\varphi}$ is the Ricci operator of $(M, g_{\varphi})$ and $\tau_{\varphi} \in \mathfrak{so}(TM)$ is the skew-symmetric operator corresponding to the torsion $2$-form $\tau_{\varphi} = g_{\varphi}(\tau_{\varphi}\cdot, \cdot)$ for closed $G_2$-structures. (Note: $T$ from [Proposition 9.4, \cite{LW17}] is equal to $-\frac{1}{2}\tau_\varphi$.) The operator $Q_\varphi \in \text{sym}(TM)$ satisfies $\theta(Q_\varphi)\varphi = \Delta_{\varphi}\varphi$ where $\theta: \mathfrak{gl}(\mathfrak p) \rightarrow \text{End}(\Lambda^3{\mathfrak p}^*)$ is the representation given by $$\theta(A)\psi = - \psi(A\cdot, \cdot, \cdot) - \psi(\cdot, A\cdot, \cdot) - \psi(\cdot, \cdot, A\cdot)\, \, \, \, \, \, \, \, \, \, A \in \mathfrak {gl}(\mathfrak p), \, \, \psi \in \Lambda^3\mathfrak p^*$$ obtained as the derivative of the natural left $\GL(\mathfrak p)$-action on $3$-forms: $h \cdot \psi = \psi(h^{-1}\cdot, h^{-1}\cdot, h^{-1}\cdot)$ (see \cite{Lau17a} for more details). It turns out $Q_\varphi$ coincides with $-h$ in (\ref{eq:2.2}). Setting $q_\varphi:= g_\varphi(Q_\varphi \cdot,\cdot)$, one obtains the Laplacian soliton equation (\ref{eq:1.2}) from equation (\ref{eq:2.3}) by similar arguments made in [LW]. Note that (\ref{eq:1.2}) corresponds to a \textit{geometric $q$-flow} of the metric $g_\varphi$ with $c = -(1/3)\lambda$ and \textit{$q$-soliton} $-2q_\varphi$. When $X = \nabla f$ for some smooth function $f:M \rightarrow \mathbb R$, equation (\ref{eq:1.2}) is the \textit{gradient Laplacian soliton equation} \begin{equation}\label{eq:2.5} \hess f = - q_\varphi - \dfrac{1}{3}\lambda g_\varphi\end{equation} with corresponding closed gradient Laplacian soliton $(\varphi, \nabla f, \lambda)$. The function $f$ is commonly referred to as the \textit{potential function}.

For more foundational material on $G_2$-structures and the Laplacian flow, we refer the reader to \cites{Bry06, Kar09, Kar20} and \cite{FG82}. We note that Haskins-Nordstr\"om investigate cohomogeneity-one steady gradient Laplacian solitons with symmetry groups $\Sp(2)$ and $\SU(3)$ in \cite{HN21}. Examples of gradient Laplacian solitons in the non-homogeneous setting are referenced in \cite{HN21}. A recent preprint \cite{HKP22} by Haskins-Kahn-Payne shows uniqueness of asymptotically conical gradient Laplacian solitons. We also mention that Garrone in \cite{Gar22} studies closed $G_2$-structures in the setting of isometric flow, where the critical points are $G_2$-structures with divergence-free full torsion tensor.

\subsection{The Structure Theorem}

If a $G_2$-structure $(M, \varphi)$ admits a gradient Laplacian soliton $(\varphi, \nabla f, \lambda)$, then the triple satisfies the gradient Laplacian soliton equation (\ref{eq:2.5}). To be consistent with notation of \cite{PW22}, we set  \begin{equation}\label{eq:2.6} q: = -q_\varphi - \dfrac{1}{3}\lambda g_\varphi,\end{equation} where $q$ is a $G$-invariant symmetric $2$-tensor. Observe that the gradient Laplacian soliton equation being satisfied by $(\varphi,\nabla f, \lambda)$ where $\nabla f \ne 0$ is equivalent to there being a non-constant $f \in F(q) = \{\hess f = q\}$ as studied in \cite{PW22}. Petersen-Wylie's motivation for studying the solution space $F(q)$ is due to the equation $\hess f = q$ arising naturally from gradient solitons for geometric flows where the tensor $q$ involves the curvature of the manifold. We remark that if $f$ is constant, then $X = \nabla f = 0$ and so the gradient solitons would be of the form $(\varphi, 0, \lambda)$, i.e., they are \textit{trivial} gradient solitons and correspond to \textit{eigenforms} satisfying $\Delta_\varphi \varphi = \lambda \varphi$. We call gradient solitons where $\nabla f \ne 0$ \textit{non-trivial} solitons. We consider non-trivial gradient solitons in this paper.

To obtain the structure theorem for closed gradient Laplacian solitons on homogeneous spaces, we need the following theorem of Petersen-Wylie ([Theorem 3.6, \cite{PW22}]).

\begin{theorem}[Structure Theorem of Petersen-Wylie] Let $(M, g)$ be a $G$-homogeneous manifold and let $q$ be a $G$-invariant symmetric $2$-tensor. If $f \in F(q)$ is a non-constant function, then either
\begin{enumerate}
    \item $(M, g)$ is isometric to a product $N \times \mathbb R^k$ where $f$ is constant on $N$
    
    \item $(M, g)$ is a $1$-dimensional extension, $g = dr^2 + g_r$, and $f(x, y) = ar + b$,
    
    \item $(M, g)$ is isometric to a product $N \times \mathbb R^k$ where $N$ is a $1$-dimensional extension and $f(x, y) = ar(x) + v(y)$, where $v$ is a function on $\mathbb R^k$ and $r$ is a distance function on $N$. 
\end{enumerate}
\end{theorem}

\noindent The three possible structures depend on the divergence of the $G$-invariant symmetric $2$-tensor $q$. That is, if $q$ is divergence-free, then we are in case (1) ([Theorem 1.1, \cite{PW22}]). The converse also holds: if the product $N \times \mathbb R^k$ has a non-constant function $f \in F(q)$ that is constant on $N$, then $q$ is divergence-free. To see this, we use the Bochner formula $\Div(\nabla \nabla f) = \ric(\nabla f) + \nabla \Delta f$ from [Lemma 2.1, \cite{PW09}]. Since $f$ is a function on the Euclidean factor only, $\nabla f \in T_p\mathbb R^k$. It follows that $\nabla f \in \ker\ric$, hence $\ric(\nabla f) = 0$. Moreover, $\Delta f = \tr \nabla \nabla f = \tr\hess f = \tr q $ is constant as $M$ is homogeneous and $q$ is invariant. Thus $\Div q = \Div \hess f = \Div(\nabla \nabla f) = 0$. If $q$ is not divergence-free, then the structure of $M$ can be as in either case (2) or (3). So to apply Theorem 2.1, we must compute the divergence of $q = - q_{\varphi} - \frac{1}{3}\lambda g_\varphi$.

\begin{proof}[Proof of Theorem 1.1]
Let $(M, \varphi)$ be a closed homogeneous $G_2$-structure.  By [Proposition 2.2, \cite{Lau17a}], $$\scal_\varphi = - \dfrac{1}{2}|\tau_\varphi|^2 \, \, \, \, \, \text{and} \, \, \, \, \, |\tau_\varphi|^2 = -\dfrac{1}{2}\tr \tau_\varphi^2.$$ Putting these two equations together yields $$-\dfrac{1}{3}\scal_\varphi = -\dfrac{1}{12}\tr \tau_\varphi^2.$$ The expression for  $Q_\varphi$ from (\ref{eq:2.4}) can be written \begin{equation}\label{eq:2.7} Q_\varphi = \ric_\varphi -  \dfrac{1}{3} \scal_\varphi I + \dfrac{1}{2}\tau_\varphi^2.\end{equation} Taking the divergence gives \begin{equation}\label{eq:2.8}
    \Div Q_\varphi  =  \dfrac{1}{2}\Div\tau_\varphi^2,
\end{equation} where we used the 2nd contracted Bianchi identity, $\Div \ric_\varphi = \frac{1}{2}D\scal_\varphi$, along with the fact that $\scal_\varphi$ is constant on homogeneous spaces. Performing a type change of equation (\ref{eq:2.6}) to $(1,1)$-tensors with respect to $g_\varphi$, we get $Q = -Q_\varphi - \frac{1}{3}\lambda I$. Then $\Div Q = - \frac{1}{2}\Div\tau_{\varphi}^2$, or equivalently, \begin{equation}\label{eq:2.9} \Div q = - \dfrac{1}{2}\Div \tau_{\varphi}^2, \end{equation} where $\tau_{\varphi}^2 = g_\varphi(\tau_{\varphi}^2 \cdot, \cdot)$. So to compute $\Div q$, it suffices to compute $\Div \tau_{\varphi}^2$.

Suppose $(\varphi,\nabla f, \lambda)$ is a gradient Laplacian soliton where $\nabla f \ne 0$. Then equation (\ref{eq:2.5}) is satisfied and implies that there is a non-constant $f \in \{\hess f = q\}$. It is clear that $q$ is a symmetric $2$-tensor as both $q_\varphi$ and $(1/3)\lambda g_\varphi$ are symmetric. That $q$ is $G$-invariant follows from $\varphi$ being $G$-invariant. More precisely, for $\gamma \in G$, we have $\gamma^*\varphi = \varphi$ and $\gamma^*g_{\varphi} = g_\varphi$. It follows from isometry invariance of the Ricci tensor that $\gamma^*\ric(g_\varphi) = \ric(\gamma^*g_\varphi) = \ric(g_\varphi)$ as $\gamma \in G \subset \Aut(M, \varphi) \subset \Iso(M, g_\varphi)$. Moreover, since the torsion $2$-form $\tau_\varphi$ is determined by $\varphi$, $\gamma^*\tau_{\varphi} = \tau_{\gamma^*\varphi} = \tau_{\varphi}$. Thus $q$ is $G$-invariant. Combining these observations with the preceding discussion, we apply Theorem 2.2 to obtain the possible structures determined by whether $\tau_\varphi^2$ is divergence-free or not. 
\end{proof}

\subsection{On computing $\Div\tau_\varphi^2$}

The following are several useful lemmas for computing $\Div \tau_\varphi^2$. We first state a lemma regarding the divergence of general $2$-tensors.

\begin{lemma}
Let $(M, g)$ be a Riemannian manifold and $(e_i)_i$ an orthonormal basis on $T_pM$. For any $(0, 2)$-tensor of the form $T(\cdot, \cdot) = g(A \cdot, \cdot)$, where $A$ is its $(1,1)$-dual tensor with respect to $g$, we have \begin{equation}\label{eq:2.10}(\Div T)(U) = \sum_{i = 1}^7 g(\nabla_{e_i}(A(e_i)) - A(\nabla_{e_i}e_i), U).\end{equation}
\end{lemma}

\noindent We will refer to the sums $\sum_i  g(\nabla_{e_i}(A(e_i)), \cdot)$ and $\sum_i g(A(\nabla_{e_i}e_i), \cdot)$ from formula (\ref{eq:2.10}) as (\ref{eq:2.10}a) and (\ref{eq:2.10}b), respectively.

\begin{remark} Note $\Div q$ is a $(0, 1)$-tensor. When $A$ is symmetric (hence $T$ is symmetric) it is not hard to show $(\nabla_{e_i}T)(e_i, U) = (\nabla_{e_i}T)(U, e_i)$ for any vector $U$. 
\end{remark}

We make a definition that will be useful in computations.

\begin{definition}
We say that a basis $(e_i)_i$ for $\mathfrak g$ is \textit{orthogonally nice} if $$[e_i, e_j] = ce_k \, \, \, \, \, \& \, \, e_i, e_j \perp e_k.$$ If $(e_i)_i$ is an orthonormal basis, then this condition is equivalent to  $$[e_i, e_j] = ce_k \, \, \, \, \, \& \, \, e_i, e_j \ne e_k.$$ 
The motivation for defining such a basis is due to it being a sufficient condition for \textit{diagonally trivial derivatives}, i.e., $$\nabla_{e_i}e_i = 0 \, \, \, \, \, \forall \, \, i,$$ provided $(e_i)_i$ is orthonormal.
\end{definition}

\begin{remark}
Our definition of an ``orthogonally nice'' basis differs from the notion of a ``nice'' basis as defined by Lauret-Will in \cite{LW13}: a basis of a Lie algebra is \textit{nice} if $[e_i, e_j]$ is always a scalar multiple of some element in the basis and $[e_i, e_j]$, $[e_r, e_s]$ can be a nonzero multiple of the same $e_k$ only if $\{i, j\} \cap \{r, s\} = \emptyset$. All of the bases $(e_i)_i$ for $(\mathfrak n_i, \varphi_i)$ for $i = 1, ..., 7$ are nice. Moreover, they are orthogonally nice, hence the structure equations for $\mathfrak n_i$ yields diagonally trivial derivatives. The characterization of nice bases is used by Lauret-Will as well as others referenced in \cite{LW13} to study nilsolitons on nilmanifolds and stably Ricci-diagonal metrics. A basis for a Lie algebra is \textit{stably Ricci-diagonal} if any diagonal left-invariant metric has diagonal Ricci tensor (see \cite{LW13}, \cite{Kri21}). One relevant fact in the nilpotent case is the following: a basis of a nilpotent Lie algebra is stably Ricci-diagonal if and only if it is nice ([Theorem 1.1, \cite{LW13}]). We note that some of the results to follow may hold with the hypotheses of nice bases on nilpotent Lie groups. We also note that Krushnan studies nice bases and diagonality of the Ricci tensor in a more general setting in \cite{Kri21}.
\end{remark}

\begin{lemma}[Consequences of the Koszul Formula]
For any orthonormal basis $(e_i)_i$, \begin{enumerate}
    \item $g(\nabla_{e_i}e_i, e_j) = -g([e_i, e_j], e_i) =  g([e_j,  e_i], e_i)$
    
    \item  $g(\nabla_{e_i}e_j, e_k) = \dfrac{1}{2}[g([e_i, e_j], e_k) - g([e_i, e_k],e_j) -  g([e_j, e_k], e_i)]$
    
    \item $\sum_i g(\nabla_{e_i}e_i, e_j) = \tr(\ad_{e_j})$
    
    \item If $(e_i)_i$ is orthogonally nice, then $\nabla_{e_i}e_i = 0 \, \, \, \, \, \forall i.$
\end{enumerate} 
\end{lemma}

\begin{proof}
(1) and (2) follow from the Koszul formula, $(e_i)_i$ being an orthonormal basis, and skew-symmetry of the Lie bracket. (3) follows from (1), $\ad_{e_j}(e_i) = [e_j, e_i]$, and the definition of trace. (4) follows from (1) and $(e_i)_i$ being orthogonally nice. 
\end{proof}

\begin{proposition}
Let $(\mathfrak g, [\cdot, \cdot])$ be the Lie algebra of a Lie group $G$ with closed $G_2$-structure $\varphi$.
If $\tau_{\varphi}^2$ is diagonal with respect to an orthogonally nice orthonormal basis $(e_i)_i$, then $\tau_\varphi^2$ is divergence-free, hence $Q_\varphi$ is divergence-free. Moreover, if $\ric_\varphi$ is also diagonal with respect to $(e_i)_i$, then $Q_\varphi$ is diagonal if and only if $\tau_\varphi^2$ is.
\end{proposition}

\begin{proof} Since $(e_i)_i$ is an orthogonally nice orthonormal basis, we have $\nabla_{e_i}e_i = 0$ $\forall$ $i$. So if $\tau_{\varphi}^2$ is diagonal, then $\tau_{\varphi}^2(e_i) = a_ie_i$ and we get $$\nabla_{e_i}(\tau_{\varphi}^2(e_i)) = \nabla_{e_i}(a_ie_i) = a_i\nabla_{e_i}e_i = 0 \, \, \, \, \, \forall \, \, i.$$ Hence the sum (\ref{eq:2.10}a) = 0. Moreover, diagonally trivial derivatives implies the sum (\ref{eq:2.10}b) = 0. Thus $\Div\tau_{\varphi}^2 = 0$. The last statement follows easily from (\ref{eq:2.7}).
\end{proof}

\begin{remark}
The converse of Proposition 2.8 is not true; in the case of $\mathfrak n_4$, $\Div\tau_{\varphi_4}^2 = 0$ while $\tau_{\varphi_4}^2$ is not diagonal (see Section~\ref{sec:3}). 
\end{remark}

We now state a key lemma used in the proof of the non-divergence-free cases of Theorem 1.6. This key lemma is an instance of [Proposition 3.1, \cite{Gri21}]. We also include [Corollary 3.2, \cite{Gri21}] as it will be used to prove some cases of Theorem 1.6.

\begin{lemma}[Key Lemma]
Let $(M, \varphi)$ be a closed $G_2$-structure. For any gradient Laplacian soliton $(\varphi, \nabla f, \lambda)$, we have 
\begin{equation}\label{eq:2.11} g(\ric(\nabla f), \cdot) = -\dfrac{1}{2}\Div \tau_{\varphi}^2 + \nabla \tr q_{\varphi}.\end{equation} If in addition $\tr q_{\varphi}$ is constant (e.g., when $M$ is homogeneous)  then \begin{equation}\label{eq:2.12} g(\ric(\nabla f), \cdot) = -\dfrac{1}{2}\Div\tau_{\varphi}^2.\end{equation}
\end{lemma}

\begin{proof} The gradient Laplacian soliton equation (\ref{eq:2.5}) type changed to $(1,1)$-tensors is \begin{equation}\label{eq:2.13}\nabla \nabla f = -Q_{\varphi} - \dfrac{1}{3}\lambda I.\end{equation} Taking the divergence of (\ref{eq:2.13}) and using the Bochner formula, $\Div \nabla \nabla f = \ric(\nabla f) + \nabla \Delta f$ ([Lemma 2.1, \cite{PW09}]), gives \begin{equation}\label{eq:2.14}\ric(\nabla f) + \nabla \Delta f = -\Div Q_{\varphi}.\end{equation} On the other hand, taking the trace of (\ref{eq:2.13}) gives \begin{equation}\label{eq:2.15}\Delta f = -\tr Q_{\varphi} -\dfrac{7}{3}\lambda .\end{equation} Substituting (\ref{eq:2.15}) into (\ref{eq:2.14}) yields \begin{equation}\label{eq:2.16}\ric(\nabla f) = -\Div Q_\varphi + \nabla \tr Q_{\varphi}.\end{equation} Combining (\ref{eq:2.8}) and (\ref{eq:2.16}) yields (\ref{eq:2.11}). If $\tr q_\varphi$ is constant, $\nabla \tr q_{\varphi} = 0$ and we get (\ref{eq:2.12}). The fact that $\tr q_\varphi$ is constant on homogeneous spaces follows from observing that it is a constant multiple of $\scal_\varphi$, which is constant on homogeneous spaces.
\end{proof}

\begin{corollary}[{[Corollary 3.2, \cite{Gri21}]}]
For any constant trace, divergence-free $2$-tensor $q$, the gradient solitons of its flow has the property that $\ric(\nabla f) = 0$. 
\end{corollary}

\begin{definition}
A homogeneous $G_2$-structure $(M, \varphi)$ is \textit{Laplacian flow diagonal} if the $\Aut(M, \varphi)$-invariant Laplacian flow solution $\varphi(t)$ starting at $\varphi$ satisfies the following property: at some $p \in M$, there is an orthonormal basis $\beta$ with respect to  $\left<\cdot, \cdot\right>_{\varphi}$ at $T_pM$, such that $Q_\varphi(t)$ is diagonal with respect to $\beta$ for all $t$.
\end{definition}

\begin{remark}
For homogeneous Laplacian solitons, $(M = G/K, \varphi)$ being Laplacian flow diagonal is equivalent to it being an algebraic soliton (see [Theorem 4.10, \cite{Lau17a}]). 
\end{remark}

\begin{corollary}
Let $G$ be a Lie group with closed $G_2$-structure $\varphi$ that is Laplacian flow diagonal with respect to an orthogonally nice orthonormal basis $(e_i)_i$. Suppose $\ric_{\varphi}$ is diagonal with respect to $(e_i)_i$. 
\begin{enumerate} 
    \item If $(\varphi, \nabla f, \lambda)$ is a gradient Laplacian soliton, then $G$ must be a product metric $\mathbb R^k \times N$ with $f$ constant on $N$.
    
    \item If in addition the kernel of the Ricci tensor is trivial, then $\varphi$ is not a gradient soliton.
\end{enumerate}
\end{corollary}

\begin{proof}
By the last statement of Proposition 2.8, $\tau_\varphi^2$ is diagonal. Since $(e_i)_i$ is an orthogonally nice orthonormal basis, we get $\Div \tau_{\varphi}^2 = 0$ by Proposition 2.8. Thus (1) follows from the Structure Theorem.  To show (2), note that the Key Lemma gives that $\ric_\varphi(\nabla f) = 0$. Since the kernel of $\ker\ric_\varphi = 0$, it must be that $\nabla f = 0$. Hence $f$ is constant, a contradiction. 
\end{proof}

\begin{remark} Corollary 2.14 can be useful in determining the structure of a homogeneous closed gradient Laplacian soliton without having to compute $\Div\tau_\varphi^2$ explicitly.
\end{remark}

\subsection{Some related consequences of the gradient soliton equation}

\begin{definition}
$G_2$-structures $(\mathfrak g_1, \psi_1)$ and $(\mathfrak g_2, \psi_2)$ are said to be \textit{equivalent}, denoted $(\mathfrak g_1, \psi_1) \simeq (\mathfrak g_2, \psi_2)$, if there is a Lie algebra isomorphism $h: \mathfrak g_1 \rightarrow \mathfrak g_2$ such that $h\cdot \psi_1 = \psi_2$. Moreover, we say that $G_2$ structures are \textit{homothetic} if there is a $c \in \mathbb R^*$ such that $(\mathfrak g_1, \psi_1) \simeq (\mathfrak g_2, c\psi_2)$. 
\end{definition}

We show that if two $G_2$-structures on the same Lie algebra are equivalent or homothetic, then one is a gradient Laplacian soliton if and only if the other is. This is needed for Theorem 1.6.

\begin{proposition}
If $\psi_1, \psi_2$ are positive and if either $(\mathfrak g, \psi_1) \simeq (\mathfrak g, \psi_2)$ or $(\mathfrak g, \psi_1) \simeq (\mathfrak g, c\psi_2)$ for some $c \in \mathbb R^*$, then $\psi_1$ is a gradient Laplacian soliton if and only if $\psi_2$ is.
\end{proposition}

\begin{proof}
For any diffeomorphism $\varphi \in \Diff(M)$, tensor $T$, and vector field $X$, we have $$\varphi^*(\mathcal L_X T) = \mathcal L_{\varphi^*X}(\varphi^*T)$$ (see exercise 1.23 in \cite{CLN06}). Also, if $f: M \rightarrow \mathbb R$, we have $$\varphi^*(\nabla^g f) = \nabla^{\varphi^* g}(f \circ \varphi).$$ Suppose $(\mathfrak g, \psi_1) \simeq (\mathfrak g, \psi_2)$ and $\psi_2$ is a gradient Laplacian soliton, i.e., $\Delta_{\psi_2}\psi_2 = \lambda\psi_2 + \mathcal L_{\nabla^{g_{\psi_2}} f}\psi_2$ for some potential function $f$. Since $(\mathfrak g, \psi_1) \simeq (\mathfrak g, \psi_2)$, there is a Lie algebra isomorphism $h:\mathfrak g \rightarrow \mathfrak g$ in $\Aut(\mathfrak g)$ such that $h\cdot \psi_2 = \psi_1$ [Note: $h \in \Diff(\mathfrak g)$ as any linear isomorphism of vector spaces is smooth]. For any geometric structure $\gamma$, $h \cdot \gamma = (h^{-1})^*\gamma$. Moreover, [Lemma 2.2 (ii)(a), \cite{Nic18}] states that for any $h \in \Aut(\mathfrak g)$, $\Delta_{h \cdot \psi}h \cdot \psi = h\cdot \Delta_\psi \psi.$ Putting these together, we get
\begin{align*}\Delta_{\psi_1}\psi_1 &= \Delta_{h \cdot \psi_2}(h \cdot \psi_2) = h \cdot \Delta_{\psi_2}\psi_2  = h \cdot (\lambda\psi_2 + \mathcal L_{\nabla^{g_{\psi_2}} f}\psi_2)  \\ &= \lambda (h \cdot \psi_2) + h \cdot(\mathcal L_{\nabla^{g_{\psi_2}}f}\psi_2) =  \lambda \psi_1 + (h^{-1})^*(\mathcal L_{\nabla^{g_{\psi_2}} f}\psi_2) \\ &= \lambda \psi_1 + \mathcal L_{(h^{-1})^*(\nabla^{g_{\psi_2}} f)}((h^{-1})^*\psi_2) = \lambda \psi_1 + \mathcal L_{\nabla^{(h^{-1})^*g_{\psi_2}} (f\circ h^{-1})}(h \cdot \psi_2) \\ &= \lambda \psi_1 + \mathcal L_{\nabla^{g_{\psi_1}} (f\circ h^{-1})}\psi_1, \end{align*} where in the last equality we used $ (h^{-1})^*g_{\psi_2} = h \cdot g_{\psi_2} = g_{h \cdot \psi_2} = g_{\psi_1}$ for any $h \in \GL(\mathfrak g)$. Thus $\psi_1$ is also a gradient Laplacian soliton. If instead $\psi_1$ is a gradient Laplacian soliton, the same argument with $h^{-1}$ in place of $h$ gives that $\psi_2$ is also a gradient soliton.

Suppose $(\mathfrak g, \psi_1) \simeq (\mathfrak g, c\psi_2)$ for some $c \in \mathbb R^*$ and that $\psi_1$ is a gradient Laplacian soliton. Since $(\mathfrak g, \psi_1) \simeq (\mathfrak g, c\psi_2)$, there is some Lie algebra isomorphism $h: \mathfrak g \rightarrow \mathfrak g$ in $\Aut(\mathfrak g)$ such that $h \cdot \psi_1 = c\psi_2$. By [Lemma 2.2 (ii)(b), \cite{Nic18}], $\Delta_{c\psi} c\psi = c^{\frac{1}{3}}\Delta_\psi\psi$. Then \begin{align*}
    c^{\frac{1}{3}}\Delta_{\psi_2}\psi_2 &= \Delta_{c\psi_2}(c\psi_2) = \Delta_{h \cdot \psi_1}(h\cdot \psi_1) = h\cdot (\Delta_{\psi_1}\psi_1) = h \cdot(\lambda\psi_1 + \mathcal L_{\nabla^{g_{\psi_1}} f}\psi_1) \\ &= \lambda (h \cdot \psi_1) + (h^{-1})^*(\mathcal L_{\nabla^{g_{\psi_1}} f}\psi_1) = c\lambda \psi_2 + \mathcal L_{(h^{-1})^* (\nabla^{g_{\psi_1}}f)}  ((h^{-1})^* \psi_1) \\ & = c\lambda \psi_2 + \mathcal L_{ \nabla^{(h^{-1})^*g_{\psi_1}}(f \circ h^{-1})}  (h \cdot \psi_1) = c\lambda \psi_2 + \mathcal L_{ \nabla^{c^{2/3}g_{\psi_2}}(f\circ h^{-1})} (c \psi_2) \\ &= c\lambda \psi_2 + c\mathcal L_{ \nabla^{g_{\psi_2}}(f\circ h^{-1})}  \psi_2,
\end{align*} where we used [Lemma 2.1(iii), \cite{Nic18}] $$(h^{-1})^* g_{\psi_1} = h\cdot g_{\psi_1} = g_{h \cdot \psi_1} = g_{c\psi_2} = c^{2/3}g_{\psi_2}$$ in the second to last equality and $\nabla^{c^{2/3}g_{\psi_2}} = \nabla^{g_{\psi_2}}$ as $c^{2/3} > 0$ in the last. So $$   c^{\frac{1}{3}}\Delta_{\psi_2}\psi_2 = c\lambda \psi_2 + c\mathcal L_{ \nabla^{g_{\psi_2}}(f\circ h^{-1})}  \psi_2$$ if and only if $$ \Delta_{\psi_2}\psi_2 = c^{\frac{2}{3}}\lambda \psi_2 + \mathcal L_{ \nabla^{g_{\psi_2}}(c^{\frac{2}{3}}f\circ h^{-1})}  \psi_2.$$ Thus $(\psi_2, \nabla^{g_{\psi_2}} c^{\frac{2}{3}}(f \circ h^{-1}), c^{\frac{2}{3}}\lambda)$ is a gradient Laplacian soliton. Similar arguments show if $\psi_2$ is a gradient soliton, then so is $\psi_1$.
\end{proof}

We include for completeness some consequences of the gradient Laplacian soliton equation (\ref{eq:2.5}) on closed $G_2$-structures (see \cites{Bry06, LW17, HN21} for more details). Note these results are immediate consequences of formulas in Section 9 of \cite{LW17}.

\begin{lemma}
Let $(M, \varphi)$ be a closed $G_2$-structure. If $(\varphi, \nabla f, \lambda)$ is a gradient Laplacian soliton, then \begin{enumerate}
    \item $\scal_{\varphi} \leq 0$
    \item $\Delta f = -\dfrac{7}{3}\lambda - \dfrac{2}{3}\scal_{\varphi}$
    \item $\nabla \Delta f = -\dfrac{2}{3}\nabla \scal_{\varphi}$
    \item $\nabla f \lrcorner T = 0$ where $T = -\frac{1}{2}\tau_\varphi$.
\end{enumerate}
\end{lemma}

\begin{proof}
(1) is well known (see [Proposition 2.2 (iii), \cite{Lau17a}] or [Corollary 2.4, \cite{LW17}]). [Corollary 2.4, \cite{LW17}] states that $\scal_{\varphi} = -|T|^2$ where $T = -\frac{1}{2}\tau_{\varphi}$ is the full torsion tensor for closed $G_2$-structures.

Taking the trace of (\ref{eq:2.5}) yields $$\Delta f = -\scal_{\varphi} + \dfrac{7}{3}\scal_{\varphi} - \dfrac{1}{2}\tr \tau_{\varphi}^2 - \dfrac{7}{3}\lambda.$$ By Proposition 2.2 (i) and (ii) of \cite{Lau17a}, $$- \dfrac{1}{2}\tr \tau_{\varphi}^2 = |\tau_{\varphi}|^2 = -2\scal_{\varphi}.$$ Substituting this back into the preceding equation and collecting the scalar curvature terms yields (2). Taking the derivative of (2) yields (3). (4) follows from the discussion in Section 9 of \cite{LW17}.
\end{proof}

\begin{corollary}
 If $(\varphi, \nabla f, \lambda)$ is a homogeneous closed gradient Laplacian soliton and $\tau_\varphi^2$ is divergence-free, then  $$\frac{1}{2}D_X\|\nabla f\|^2 = \frac{1}{3}(\scal_\varphi - \lambda)g(\nabla f, X) \, \, \, \, \, \, \, \, \, \, \forall \, \, X \in TM.$$ If in addition $\|\nabla f\| = $ constant, then $\lambda = \scal_\varphi$. Since $\scal_\varphi \leq 0$ for closed $G_2$-structures, the soliton is either shrinking or steady.
\end{corollary}

\begin{proof} The gradient Laplacian soliton equation (\ref{eq:2.5}) yields
$$\hess f (\nabla f, X) = -\ric_\varphi(\nabla f, X) - \frac{1}{2}\tau_\varphi^2(\nabla f, X) + \frac{1}{3}(\scal_\varphi - \lambda)g(\nabla f, X).$$ Since $\tau_\varphi^2$ is divergence-free, by the Key Lemma we have $\ric_\varphi(\nabla f) = -\frac{1}{2}\Div \tau_\varphi^2 = 0$. By Lemma 2.18 (4), we have $\nabla f \lrcorner \tau_\varphi = 0$ and so $\tau_\varphi^2(\nabla f, X) = g(\tau_\varphi^2(\nabla f), X) = -g(\tau_\varphi(\nabla f), \tau_\varphi(X)) = 0$. By [Proposition 3.2.1 (3), \cite{Pet16}], $\hess f(\nabla f, X) = \frac{1}{2}D_X\|\nabla f\|^2$ for all $X \in TM$. Putting these items together in the soliton equation gives the desired formula. If $\|\nabla f\| = $ constant, then the left-hand side of the formula is zero while the right-hand side is $\frac{1}{3}(\scal_\varphi - \lambda) \|\nabla f\|^2$. Since $\|\nabla f\|^2 > 0$ as $f$ is non-constant, it follows that $\lambda = \scal_\varphi$. 
\end{proof}

\begin{remark}
Without the homogeneous assumption, the formula in Corollary 2.19 is $2^{-1}D_X\|\nabla f\|^2 = - g(\nabla \tr q_\varphi, X) + 3^{-1}(\scal_\varphi - \lambda)g(\nabla f, X)$.
\end{remark}

\section{Gradient Laplacion Solitons on Nilpotent Lie Groups}\label{sec:3}

In this section we prove Theorem 1.6. Tables consisting of relevant data for each nilpotent Lie algebra $(\mathfrak n_i, \varphi_i)$ are provided. Note that $\tau_{\varphi_i}$, hence $\tau_{\varphi_i}^2$, are obtained with respect to bases and corresponding structure equations from the tables in \cite{Nic18}. We first compute the divergence of $\tau_{\varphi_i}^2$. We consider divergence-free and non-divergence-free cases separately in the proof of Theorem 1.6. Lastly, we show the closed $G_2$-structure $(\mathfrak n_{12}, \varphi_{12})$ constructed in \cite{FFM16} is not gradient.\\

\noindent \textit{Notation:} $N$ is as in the structure theorem while $N$ with a subscript, $N_i$, denotes the nilpotent Lie group with corresponding nilpotent Lie algebra $\mathfrak n_i$.

\subsection{Computing $\Div \tau_{\varphi_i}^2$ for $\mathfrak n_i$}

\begin{table}
\caption{}
\resizebox{\textwidth}{!}{%
\begin{tabular}{|c|c|c|}\hline
& $(\mathfrak n_2(1,1), \varphi_2)$ & $(\mathfrak n_3(1, 1-c, c), \varphi_3)$, $0 < c < 1/2$ \\ 
\hline
$\ric_{\varphi_i}$   & $-\Diag(1, \frac{1}{2}, \frac{1}{2}, 0, -\frac{1}{2}, -\frac{1}{2},  0)$ & $\dfrac{1}{2}\Diag(-2+2c -  c^2, -1-c^2, -1+2-2c^2, 1, (-1 + c)^2, c^2,  0)$ \\ \hline
  $\tau_{\varphi_i}$ & $-e^{35} + e^{26}$  & $-ce^{16} + (1-c)e^{25} - e^{34}$ \\ \hline
    $\tau_{\varphi_i}^2$ & $\Diag(0, -1,  -1, 0, -1, -1, 0)$ & $\Diag(-c^2, -(1 - c)^2, -1, -1, -(1-c)^2, -c^2, 0)$  \\ \hline
      $Q_{\varphi_i}$ & $\frac{1}{3}\Diag(-2,-2,-2,1,1,1)$ & $\frac{1- c + c^2}{3}\Diag(-2,-2,-2,1,1,1,1)$  \\ \hline
      $\lambda_i$ & $5$ & $5(1 - c + c^2)$ \\ \hline
\end{tabular}}
\end{table}

\begin{table}
\caption{}
\resizebox{\textwidth}{!}{%
\begin{tabular}{|c|c|c|}
\hline
& $(\mathfrak n_4(\sqrt{2},1, \sqrt{2}, 1), \varphi_4)$ & $(\mathfrak n_6(\sqrt{2},\sqrt{2},1, 1), \varphi_6)$ \\ 
\hline
$\ric_{\varphi_i}$   & $\Diag(-2, -2, \frac{1}{2}, -1,  -\frac{1}{2}, \frac{3}{2},  \frac{1}{2})$ & $\Diag(-3, -1, -1, \frac{1}{2}, \frac{1}{2}, \frac{1}{2}, \frac{1}{2})$ \\ 
\hline
  $\tau_{\varphi_i}$ & $ -\sqrt{2}e^{34} + \sqrt{2}e^{16} - e^{56} + e^{37}$ &  $ - \sqrt{2}e^{34} + \sqrt{2}e^{25} - e^{56} + e^{47}$\\ \hline
    $\tau_{\varphi_i}^2$ & $ \begin{pmatrix}
    -2 & 0& 0& 0& \sqrt{2}& 0 & 0\\
    0 &0 &0 &0 & 0& 0& 0\\
    0& 0& -3& 0& 0& 0&0 \\
    0 & 0&0 &-2 &0 &0 &\sqrt{2} \\
    \sqrt{2}& 0& 0& 0& -1& 0& 0\\
    0 &0 &0 & 0&0 &-3 &0 \\
    0 &0 &0 & \sqrt{2}& 0& 0& -1\\
    \end{pmatrix}$ & $\begin{pmatrix} 
0& 0&0 & 0& 0& 0& 0\\
0&-2 & 0& 0& 0& -\sqrt{2}& 0 \\
0& 0& -2& 0&0 &0  &-\sqrt{2} \\
0& 0& 0& -3 & 0& 0& 0\\
0&0 &0 &0 &-3 &0 & 0\\
0& -\sqrt{2}&0 & 0& 0&-1 &0 \\
0& 0& -\sqrt{2}&0 & 0& 0&-1 \\
\end{pmatrix}$  \\ \hline
      $Q_{\varphi_i}$ & $\begin{pmatrix} -2 & 0& 0& 0&\frac{\sqrt{2}}{2} & 0& 0\\
      0& -1& 0& 0& 0& 0& 0\\ 
      0& 0& 0& 0& 0& 0& 0\\ 
      0& 0& 0& -1&0 &0 &\frac{\sqrt{2}}{2} \\ 
      \frac{\sqrt{2}}{2}&0 &0 &0 &0 &0 &0 \\ 
      0& 0& 0& 0& 0& 1& 0\\ 
      0& 0& 0& \frac{\sqrt{2}}{2}&0 & 0& 1\\ \end{pmatrix}$  & $\begin{pmatrix} -2& 0& 0& 0& -\frac{\sqrt{2}}{2}&0 & 0\\
      0&-1 &0 & 0& 0& 0& 0\\ 
      0& 0& -1& 0& 0& 0&-\frac{\sqrt{2}}{2} \\ 
      0& 0& 0& 0& 0& 0& 0\\ 
      0& 0& 0& 0& 0& 0& 0\\ 
      0& -\frac{\sqrt{2}}{2} &0 &0 &0 &1 &0 \\ 
      0& 0& -\frac{\sqrt{2}}{2}& 0& 0&0 & 1\\ \end{pmatrix}$  \\ \hline
    $\lambda_i$ & 9 &  9\\ \hline
\end{tabular}}
\end{table}

\begin{table}
\caption{}
\resizebox{\textwidth}{!}{%
\begin{tabular}{|c|c|c|}\hline
&$(\mathfrak n_5(\sqrt{2}, 1, 1, \sqrt{2}), \varphi_5)$ & $(\mathfrak n_7(-4, 2, 2, \sqrt{6}, \sqrt{6}), \varphi_7)$  \\ \hline
$\ric_{\varphi_i}$& $\Diag(-2, -2, \frac{1}{2}, -\frac{1}{2}, -1, \frac{1}{2}, \frac{3}{2})$ & $\Diag(-10, -10, 3, 11, -1, -1, -10)$ \\ \hline
$\tau_{\varphi_i}$& $\tau_{\varphi_5} = -e^{46} + e^{37} - \sqrt{2}e^{35} + \sqrt{2}e^{17}$ & $\tau_{\varphi_7} = -2e^{15} + 2e^{26} - \sqrt{6}e^{36} + \sqrt{6}e^{45} - 4e^{47},$ \\ \hline
$\tau_{\varphi_i}^2$& $\begin{pmatrix} 
-2 & 0 & - \sqrt{2} &0  &0  &0  &0 \\
0 & 0& 0& 0& 0& 0& 0\\
- \sqrt{2} & 0& - 3& 0 & 0  &0  & 0  \\
0& 0& 0& -1& 0& 0 &0 \\
0& 0& 0&0 &-2 &0  & \sqrt{2}\\
0& 0& 0& 0& 0& -1& 0\\
0& 0& 0&0 & \sqrt{2}&0 &-3 \\
\end{pmatrix}$& $\begin{pmatrix}
-4 & 0 & 0 & 2\sqrt{6} & 0 & 0 & 0  \\
0 & -4 & 2\sqrt{6} & 0 & 0 & 0 & 0\\
0 & 2\sqrt{6} & -6 & 0 & 0 & 0 & 0 \\
2\sqrt{6} & 0 & 0 & -22 & 0 & 0 & 0\\
0 & 0 & 0 & 0 & -10 & 0 & 4\sqrt{6}\\
0 & 0 & 0 & 0 & 0 & -10 & 0 \\
0 & 0 & 0 & 0 & 4\sqrt{6} & 0 & -16 \\
\end{pmatrix}$ \\ \hline
$Q_{\varphi_i}$& $\begin{pmatrix} -2 & 0 & -\frac{\sqrt{2}}{2} & 0 & 0 & 0 & 0 \\
0 & -1 & 0 & 0 & 0 & 0 & 0 \\ 
-\frac{\sqrt{2}}{2} & 0 & 0 & 0& 0 &0 &0 \\
0 & 0 & 0 & 0& 0 &0 &0 \\
0 & 0& 0 & 0 & -1 & 0 & \frac{\sqrt{2}}{2}\\
0 & 0& 0& 0& 0& 1 & 0\\
0 & 0 & 0 & 0 & \frac{\sqrt{2}}{2} & 0 & 1\end{pmatrix}$ & $\begin{pmatrix} -4 & 0 & 0 & 0 & -1 & 0 & 0 \\
0 & -4 & 0 & 0 & 0 & 1 & 0 \\ 
0 & 0 & 9 & 0& 0 &-\frac{\sqrt{6}}{2} & 0 \\
0 & 0 & 0 & 17 & \frac{\sqrt{6}}{2} & 0 & -2 \\
1 & 0 & 0 & -\frac{\sqrt{6}}{2} & 5 & 0 & 0\\
0 & -1 & \frac{\sqrt{6}}{2} & 0 & 0 & 5 & 0\\
0 & 0 & 0 & 2 & 0 & 0 & -4\end{pmatrix}$  \\ \hline
$\lambda_i$ &$9$ & $54$ \\ \hline
\end{tabular}}
\end{table}

\begin{proposition}
Let $(\mathfrak n_i, \varphi_i)$, $i = 1, ..., 7$ be the nilpotent Lie algebras admitting closed Laplacian solitons $\varphi_i$ found in \cite{Nic18}. The square of the torsion $2$-form $\tau_{\varphi_i}^2$ is divergence-free for $i = 1, 2, 3, 4, 6$ and not divergence-free for $i = 5, 7$. 
\end{proposition}

\begin{proof} The torsion $2$-form $\tau_{\varphi_1} = 0$. More precisely, the exterior derivatives obtained from trivial brackets are all $0$, hence $\tau_{\varphi_1} = - *d*\varphi_1 = 0$ regardless of what $\varphi_1$ is. It follows that $\tau_{\varphi_1}^2 = 0$, hence its divergence is $0$.

The torsion $2$-forms $\tau_{\varphi_i}$ for all other cases can be obtained via $\tau_{\varphi_i} = -* d * \varphi_i$ (see \cite{Nic18}). We obtain $\tau_{\varphi_i}^2$ from the skew-symmetric matrix representation of $\tau_{\varphi_i}$ with respect to $(e^j)_j$. We claim that when $A = \tau_{\varphi_i}^2$, the sum (\ref{eq:2.10}b) $= 0$ for each $i = 1, ..., 7$.

\begin{proof}[Proof that sum (\ref{eq:2.10}b) = 0]
Unimodular Lie groups can be characterized by the property that there is a basis $(X_j)_j$ such that $\tr(\ad_X) = \sum_j g(\ad_X(X_j), X_j) = 0$ for any $X$. As nilpotent Lie groups are unimodular, it follows that $\ad_X$ is trace-free in all cases $\mathfrak n_i$. Moreover, the Lie brackets for $\mathfrak n_i$, $i = 1, ..., 7$, are all orthogonally nice. Thus either of these two conditions imply the sum $(2.11\text{b}) = \sum_j g(A(\nabla_{e_j}e_j), \cdot) = 0$ whenever $A$ is symmetric. To see this, first note by symmetry of $A$ we have $g(A(\nabla_{e_j}e_j), \cdot) = g(\nabla_{e_j}e_j, A(\cdot))$ as $A^* = A^t = A$ over $\mathbb R$. Then for any $U = \sum_k U^ke_k$,  \begin{align*}
    \sum_j g(\nabla_{e_j}e_j, A(U)) &= \sum_j g(\nabla_{e_j}e_j, A(\sum_kU^ke_k)) = \sum_j \sum_kU^kg(\nabla_{e_j}e_j, A(e_k))\\ &= \sum_k U^k\left(\sum_jg(\nabla_{e_j}e_j,A(e_k))\right) \\ &=\sum_k U^k\left(\sum_jg(\nabla_{e_j}e_j, \sum_\ell a_{\ell k}e_\ell^k)\right) \\ &= \sum_kU^k \left(\sum_\ell\sum_ja_{\ell k} g(\nabla_{e_j}e_j, e_\ell^k) \right) = \sum_kU^k\sum_\ell a_{\ell k}\tr(\ad_{e_\ell^k}),\end{align*} where the last expression is $0$ as $\tr(\ad_{e_\ell^k}) = 0$ $\forall$ $k, \ell$. On the other hand, whenever $(e_j)_j$ is orthogonally nice, $\nabla_{e_j}e_j = 0$ for all $j$ and so $\sum_j g(A(\nabla_{e_j}e_j), U) = \sum_j g(\nabla_{e_j}e_j, A(U)) = 0$. 
\end{proof}

It remains to compute the sum (\ref{eq:2.10}a) when $A = \tau_{\varphi_i}^2$ for $i = 2,..., 7$. Computing (\ref{eq:2.10}a) when $A = \tau_{\varphi_i}^2$ amounts to computing the terms $\nabla_{e_j}(\tau_{\varphi}^2(e_j))$. This depends on both the matrix representation of $\tau_{\varphi_i}^2$ with respect to the bases $(e_j)_j$ as well as the derivatives $\nabla_{e_j}e_k$.

Since $\tau_{\varphi_2}^2$, $\tau_{\varphi_3}^2$ are diagonal and the corresponding bases are orthogonally nice, by Proposition 2.8 both $\Div\tau_{\varphi_2}^2 = 0$ and $\Div \tau_{\varphi_3}^2 = 0$. We include computation of $\Div \tau_{\varphi_5}^2$. That $\Div \tau_{\varphi_i}^2 = 0$ for $i = 4, 6$ and $\Div \tau_{\varphi_7}^2(U, V) = -16\sqrt{6}g(e_2, U)$ follows from similar computations as for $\Div \tau_{\varphi_5}^2$.

\begin{remark}
The derivatives for each case $\mathfrak n_i$ are obtained from the Koszul formula and the structure equations as in the tables of \cite{Nic18}. [Lemma 3.10, \cite{Nic18}] states that for $\mathfrak n_5(a, b, c, d)$, where $a, b, c, d$ are the structure constants, $\varphi_5$ is closed if and only if $a = d$ and $b = c$. The lemma further states that if $a^2 = 2b^2$, then $(\mathfrak n_5(a, b, b, a), \varphi_5)$ is a semi-algebraic soliton, hence is a Laplacian soliton. We prove the result for $(b = 1, a = \sqrt{2})$ and note that it holds for general $(a, b)$ where $a^2 = 2b^2$ by scaling. We do the same for all other cases $\mathfrak n_i$.
\end{remark}

\begin{center} Table of derivatives for $\mathfrak{n_5}(\sqrt{2}, 1, 1,\sqrt{2})$ \end{center}

\begin{center}
\scalebox{0.8}{\begin{tabular}{|c|c|c|c|c|c|c|c|} \hline
$\nabla_{e_i}e_j$ &1 & 2& 3& 4& 5&6 & 7\\ \hline
1& 0& $-\dfrac{\sqrt{2}}{2}e_3$ & $\dfrac{\sqrt{2}}{2}e_2 - \dfrac{1}{2}e_6$ & $-\dfrac{1}{2}e_7$ & 0 & $\dfrac{1}{2}e_3$ & $\dfrac{1}{2}e_4$ \\ \hline
2& $\dfrac{\sqrt{2}}{2}e_3$ & 0 & $-\dfrac{\sqrt{2}}{2}e_1$ & 0 & $-\dfrac{\sqrt{2}}{2}e_7$ & 0 & $\dfrac{\sqrt{2}}{2}e_5$\\ \hline
3&$\dfrac{\sqrt{2}}{2}e_2 + \dfrac{1}{2}e_6$ & $-\dfrac{\sqrt{2}}{2}e_1$ & 0 & 0& 0& $-\dfrac{1}{2}e_1$ & 0\\ \hline
4& $\dfrac{1}{2}e_7$ & 0 & 0 & 0 & 0 & 0 &$-\dfrac{1}{2}e_1$ \\ \hline
5& 0 & $\dfrac{\sqrt{2}}{2}e_7$ & 0 & 0 & 0 & 0 & $-\dfrac{\sqrt{2}}{2}e_2$ \\ \hline
6& $\dfrac{1}{2}e_3$ & 0 & $-\dfrac{1}{2}e_1$ & 0 & 0 & 0 & 0 \\ \hline
7& $\dfrac{1}{2}e_4$ & $\dfrac{\sqrt{2}}{2}e_5$ & 0 & $-\dfrac{1}{2}e_1$ & $-\dfrac{\sqrt{2}}{2}e_2$ & 0 & 0 \\ \hline
\end{tabular}}
\end{center}

\noindent \textit{Case $(\mathfrak n_5, \varphi_5)$.} We compute each term of the sum (2.5a):

$\nabla_{e_1}(\tau_{\varphi_5}^2(e_1)) = \nabla_{e_1}(-2e_1 - \sqrt{2}e_3) = -\sqrt{2}(\dfrac{\sqrt{2}}{2}e_2 - \dfrac{1}{2}e_6) = -e_2 + \dfrac{\sqrt{2}}{2}e_6$;

$\nabla_{e_2}(\tau_{\varphi_5}^2(e_2)) = \nabla_{e_2}(0) = 0$;

$\nabla_{e_3}(\tau_{\varphi_5}^2(e_3)) = \nabla_{e_3}(-\sqrt{2}e_1 - 3e_3) = -\sqrt{2}(\dfrac{\sqrt{2}}{2}e_2 + \dfrac{1}{2}e_6) = - e_2 - \dfrac{\sqrt{2}}{2}e_6$;

$\nabla_{e_4}(\tau_{\varphi_5}^2(e_4)) = \nabla_{e_4}(-e_4)= 0$;

$\nabla_{e_5}(\tau_{\varphi_5}^2(e_5)) = \nabla_{e_5}(-2e_5 + \sqrt{2}e_7) = \sqrt{2}(-\dfrac{\sqrt{2}}{2}e_2) =  -e_2$;

$\nabla_{e_6}(\tau_{\varphi_5}^2(e_6)) = \nabla_{e_6}(-e_6) =  0$;

$\nabla_{e_7}(\tau_{\varphi_5}^2(e_7)) = \nabla_{e_7}(\sqrt{2}e_5 - 3e_7) =  \sqrt{2}(-\dfrac{\sqrt{2}}{2}e_2) = -e_2$. 

\noindent Thus $$\Div \tau_{\varphi_5}^2(U, V) = \sum_{i = 1}^7 g(\nabla_{e_i}(\tau_{\varphi_5}^2(e_i)), U) = g(-4e_2, U)= -4g(e_2, U),$$ which is nonzero whenever the $e_2$ component of $U$ is nonzero.
\end{proof}

We now prove Theorem 1.6.

\subsection{Divergence-free cases: $\Div \tau_{\varphi_i}^2 = 0$}

\begin{proof}[Proof of Theorem 1.6 Case $(\mathfrak n_1, \varphi_1)$] The Lie brackets $[\cdot, \cdot]$ with respect to orthonormal basis $(e_i)_{i =1 }^7$ for $\mathfrak n_1$ are trivial, hence the covariant derivatives $\nabla_{e_i}e_j$ are trivial. So for some closed $G_2$-structure $\varphi_1$, $\ric_{\varphi_1}$, $\tau_{\varphi_1}$, and $\scal_{\varphi_1}$ are $0$. Since $\ric_{\varphi_1} = 0$ and $N_1$ is homogeneous, it follows the space is flat. Suppose $(\varphi_1, \nabla f, \lambda_1)$ is a gradient Laplacian soliton. Since $\Div \tau_{\varphi_1}^2 = 0$, the Structure Theorem yields $N_1 = N \times \mathbb R^k$ where $f$ is constant on $N$. Note $\nabla f \in T_p\mathbb R^k \subseteq\ker(\ric_{\varphi_1}) = \Span(e_i)_{i = 1}^7 = T_p\mathbb R^7$, i.e., $\nabla f$ can be written as a linear combination of elements from $(e_i)_{i = 1}^7$ and $k \leq 7$. The gradient Laplacian soliton equation $$\hess f = -\dfrac{1}{3}\lambda_1 g$$ is diagonal with respect to basis $(e_i)_{i = 1}^7$ and so $\hess f$ must also be diagonal, i.e., $\nabla_i\nabla_j f = 0$ whenever $i \ne j$. Equating matrix entries, we get $\nabla_i\nabla_i f = -\dfrac{\lambda_1}{3}$ for each $i$ and so the potential function $f$ must be of the form \begin{align*}f(x, y, z, s, u, v, w) &= -\dfrac{\lambda_1}{6}(x^2 + y^2 +z^2 + s^2 + u^2 + v^2 + w^2) \\ &-( \alpha_1x + \alpha_2y + \alpha_3z + \alpha_4s + \alpha_5u + \alpha_6v + \alpha_7w) - \beta \end{align*} which is a Gaussian soliton; $(x,y,z,s,u,v,w)$ are coordinates with respect to $(e_i)_{i = 1}^7$.
\end{proof}

\begin{proof}[Proof of Theorem 1.6 Case $(\mathfrak n_2(1,1), \varphi_2)$.] By Proposition 3.1 $\Div\tau_{\varphi_2}^2 = 0$ and so by the Structure Theorem $N_2 = N \times \mathbb R^k$ where $f$ is constant on $N$. Note $\nabla f \in T_p\mathbb R^k \subset \ker\ric_{\varphi_2} = \Span\{e_4, e_7\}$ and so $k \leq 2$. In an appropriate basis $\mathcal B$, $\hess f\big|_N = 0$ and so the restriction of the gradient Laplacian soliton equation to $N$ with respect to $\mathcal B$ becomes $q_{\varphi_2}\big|_N = - \frac{1}{3}\lambda_2 g_N.$ But this means $-\frac{2}{3} = q_{\varphi_2}\big|_N(e_1, e_1) = q_{\varphi_2}\big|_N(e_6, e_6) = \frac{1}{3},$ a contradiction. Thus $(\varphi_2, X, \lambda_2)$ cannot be gradient Laplacian soliton. 
\end{proof}

\begin{proof}[Proof of Theorem 1.6 Case $(\mathfrak n_3(1, 1-c, c), \varphi_3)$] $(\varphi_3, X, \lambda_3)$ cannot be a gradient soliton by similar arguments as in case $(\mathfrak n_2, \varphi_2)$.
\end{proof}

\begin{proof}[Proof of Theorem 1.6 Case $(\mathfrak n_4(\sqrt{2},  1,\sqrt{2}, 1), \varphi_4)$] Suppose $(\varphi_4, \nabla f, \lambda_4)$ is a gradient Laplacian soliton. By Proposition 3.1 $\Div \tau_{\varphi_4}^2 = 0$. In the context of a $(-2q_{\varphi_4})$-flow, we get $-2q_\varphi$ is also divergence-free. Furthermore, $\tr (-2q_{\varphi_4})$ is constant as $N_4$ is homogeneous. We apply Corollary 2.11 to the $(-2q_{\varphi})$-flow to get the potential function $f$ satisfies $\ric_{\varphi_4}(\nabla f) = 0$. But $\ric_{\varphi_4}$ has trivial kernel and so $\nabla f = 0$. Thus $f$ is constant, a contradiction. Therefore $(\mathfrak n_4, X, \lambda_4)$ cannot be a gradient Laplacian soliton.
\end{proof}

\begin{proof}[Proof of Theorem 1.6 Case $(\mathfrak n_6(\sqrt{2},\sqrt{2}, 1,1), \varphi_6)$] By analogous arguments as in the proof of case $(\mathfrak n_4, \varphi_4)$, we get that $(\mathfrak n_6, X, \lambda_6)$ cannot be a gradient Laplacian soliton.
\end{proof}

\subsection{Non-divergence-free cases: $\Div \tau_{\varphi_i}^2 \ne 0$}

\begin{proof}[Proof of Theorem 1.6 Case $(\mathfrak n_5, \varphi_5)$]
Suppose $(\varphi_5,\nabla f, \lambda_5)$ is a gradient Laplacian soliton. Since $\Div\tau_{\varphi_5}^2 \ne 0$, $(N_5, \varphi_5)$ has either structure 2(a) or 2(b). As $\ric_{\varphi_5}$ has trivial kernel, $N_5$ cannot split as a product and so the structure must be as in 2(a).

Suppose $(N_5, \varphi_5)$ is a one-dimensional extension where $f = ar + b$. By the Key Lemma the potential function $f$ satisfies $$g(\ric(\nabla f), \cdot) =  -\dfrac{1}{2}\Div \tau_{\varphi_5}^2(\cdot) = 2g(e_2, \cdot),$$ where the last equality follows from $\Div \tau_{\varphi_5}^2 = -4g(e_2, \cdot)$ as computed in the proof of Proposition 3.1. Note $\ric_{\varphi_5}(\nabla f) = 2e_2$.  Since $\ric_{\varphi_5}$ is diagonal with respect to $(e_i)_i$ with nonzero diagonal entries, $\nabla f = c_2e_2$. Substituting $\ric_{\varphi_5}(\nabla f) = -2c_2e_2$ in the Key Lemma yields $c_2 = -1$, and so $\nabla f = -e_2$. Since $f = ar + b$, it follows that $e_2 = \pm \nabla r$.

Assume $\nabla r = e_2$. Applying the (1,1)-tensor version of the gradient Laplacian soliton equation (\ref{eq:2.5}) to $\nabla r = e_2$ and noting that $\hess f (\nabla r ) = a\hess r(\nabla r) = 0$, we get \begin{equation}\label{eq:3.1} 0 = - \ric_{\varphi_5}(e_2) - \dfrac{1}{2}\tau_{\varphi_5}^2(e_2) + \dfrac{1}{3}(\scal_{\varphi_5} - \lambda_5)I(e_2).\end{equation} Since $\tau_{\varphi_5}^2(e_2) = 0$, (\ref{eq:3.1}) becomes $$\ric_{\varphi_5}(e_2) = -\frac{1}{3}(\scal_{\varphi_5} - \lambda_5)I(e_2).$$  Substituting $\scal_{\varphi_5} = -3$ and $\lambda_5 = 9$ yields $$-2e_2 = \ric_{\varphi_5}(e_2) = -4I(e_2) = -4e_2,$$ from which it follows that $-2 = -4$, a contradiction. By similar arguments, we arrive at a contradiction when $\nabla r = -e_2$.
[Note: There cannot be two distinct contraction constants satisfying the soliton equation for if $(\varphi, X_1, \lambda_1)$ and $(\varphi, X_1, \lambda_2)$ both satisfy \ref{eq:1.2} and $\lambda_1 \ne \lambda_2$, then $L_Xg_\varphi = L_{X_2 - X_1}g_\varphi =  2(\lambda_2 - \lambda_1)g_\varphi = cg_\varphi$ for some nonzero constant $c \in \mathbb R$ and non-trivial vector field $X$, which would imply the space is flat.]
\end{proof}

\begin{proof}[Proof of Theorem 1.6 Case $(\mathfrak n_7, \varphi_7)$]
$(\varphi_7, X, \lambda_7)$ cannot be a gradient Laplacian soliton by analogous arguments as in case $(\mathfrak n_5, \varphi_5)$. 
\end{proof}

\begin{proof}[Some final remarks on the Proof of Theorem 1.6]
When $f$ is constant, the possible gradient Laplacian solitons are of the form $(\varphi, 0, \lambda)$. In these cases, $\hess f = 0$ and the gradient soliton equation of type $(1, 1)$ is equivalent to $Q_\varphi = -3^{-1}\Diag(\lambda,...,\lambda).$  None of the matrix expressions  $Q_{\varphi_i}$ for $i = 2, ..., 7$ satisfy this equality and thus such gradient Laplacian solitons cannot occur on $\mathfrak n_i$ for $i = 2, ..., 7$. For $i = 1$, $Q_\varphi = 0$ and so we must have $\lambda = 0$, i.e., the soliton is steady; $\varphi$ is torsion-free as $\mathfrak n_1$ has trivial structure. Thus the only non-trivial gradient solitons on $\mathfrak n_1$ are Gaussian as shown above. Lastly, the result is up to homothetic $G_2$-structures by Proposition 2.17.
\end{proof}

\begin{proposition} The closed $G_2$-structure $\varphi_{12}$ on $N_{12}$ as constructed in \cite{FFM16} is not gradient up to homothetic $G_2$-structures. 
\end{proposition}

\begin{proof}
Let $(e_i)_i$ be the basis with structure equations \begin{align*}\mathfrak n_{12} = (&0, 0, 0, \frac{3}{6}e^{12}, \frac{1}{4}e^{23} + \frac{\sqrt{3}}{12}e^{13}, -\frac{\sqrt{3}}{12}e^{23} - \frac{1}{4}e^{13},  \\ &-\frac{\sqrt{3}}{6}e^{34} + \frac{\sqrt{3}}{12}e^{25} + \frac{1}{4}e^{26} + \frac{\sqrt{3}}{12}e^{16} - \frac{1}{4}e^{15}) \end{align*} and closed $G_2$ structure given by $$\varphi_{12} = -e^{124} + e^{135} + e^{167} - e^{236} + e^{257} + e^{347} - e^{456}$$ as in \cite{FFM16}. This basis and its corresponding structure equations are obtained from the canonical one for $\mathfrak n_{12}$ (see [Theorem 3.1, \cite{FFM16}] or [Table 1, \cite{Nic18}]). The structure constants and exterior derivatives are:

$$[e_1, e_2] = -\frac{\sqrt{3}}{6}e_4, [e_2,e_3] = \frac{1}{4}e_5, [e_1, e_3] = -\frac{\sqrt{3}}{12}e_5, [e_2, e_3] = \frac{\sqrt{3}}{12}e_6, $$

$$[e_1,e_3] = \frac{1}{4}e_6, [e_3, e_4] = \frac{\sqrt{3}}{6}e_7, [e_2, e_5] = -\frac{\sqrt{3}}{12}e_7, [e_2, e_6] = -\frac{1}{4}e_7,$$

$$[e_1, e_6] = -\frac{\sqrt{3}}{12}e_7, [e_1, e_5] = \frac{1}{4}e_7.$$

and $$de^1 = de^2 = de^3 = 0, \, \, \, \, \, de^4 = \frac{\sqrt{3}}{6}e^{12}, de^5 = -\frac{1}{4}e^{23} + \frac{\sqrt{3}}{12}e^{13}, de^6 = -\frac{\sqrt{3}}{12}e^{23} - \frac{1}{4}e^{13}$$ $$de^7 = -\frac{\sqrt{3}}{6}e^{34} + \frac{\sqrt{3}}{12}e^{25} + \frac{1}{4}e^{26} + \frac{\sqrt{3}}{12}e^{16} - \frac{1}{4}e^{15}.$$

As shown in \cite{FFM16}, the basis is orthonormal with respect to the associated metric $g_{\varphi_{12}}$ and the Ricci tensor is given by $$\ric_{\varphi_{12}} = \Diag(-\frac{1}{8}, -\frac{1}{8}, -\frac{1}{8}, 0, 0, 0, \frac{1}{8}) = -\frac{1}{4}I + \frac{1}{8}D,$$ where $D = \Diag(1, 1, 1, 2, 2, 2, 3)$, i.e., $g_{\varphi_{12}}$ is a nilsoliton. In Section 4 of \cite{FFM16}, it is shown that $\mathfrak n_{12}$ is Laplacian flow diagonal with respect to $(e_i)_i$ and at $t = 0$, $\varphi_{12}(0) = \varphi_{12}$. In other words $Q_{\varphi_{12}}(t)$ is diagonal along the Laplacian flow in the time interval stated in \cite{FFM16}. In particular, $Q_{\varphi_{12}}$ is diagonal with respect to $(e_i)_i$ at $t = 0$. Hence $\tau_{\varphi_{12}}^2$ is diagonal by Proposition 2.8. The basis $(e_i)_i$ is orthogonally nice. So if $(\varphi_{12}, \nabla f, \lambda_{12})$ is a gradient Laplacian soliton, then by Corollary 2.14 (1) $(N_{12}, \varphi_{12})$ must be a product metric $N \times \mathbb R^k$ where $f$ is constant on $N$. But since $\ker \ric_{\varphi_{12}} \ne \{0\}$, we cannot use Corollary 12.14 (2). 
We compute $\tau_{\varphi_{12}}^2$:

\begin{align*}
\varphi_{12} &= -e^{124} + e^{135} + e^{167} - e^{236} + e^{257} + e^{347} - e^{456} \\ *\varphi_{12} &= e^{3567} - e^{2467} + e^{2345} + e^{1457} + e^{1346} + e^{1256} + e^{1237}  \\  d*\varphi_{12} &= -\frac{1}{2}e^{12347} - \frac{1}{2}e^{12456} \\ *d*\varphi_{12} &= - \frac{1}{2}e^{56} + \frac{1}{2}e^{37} \\  \tau_{\varphi_{12}} &= -*d*\varphi_{12} = \frac{1}{2}e^{56} - \frac{1}{2}e^{37}\end{align*} Then $$\tau_{\varphi_{12}}^2 = \Diag\left(0, 0, -\frac{1}{4}, 0, -\frac{1}{4}, -\frac{1}{4}, -\frac{1}{4}\right).$$ Since $f$ is a function on $\mathbb R^k$, $\nabla f \in T_p\mathbb R^k$ and so $\ric_{\varphi_{12}}(\nabla f) = 0$, i.e., $\nabla f \in \ker(\ric_{\varphi_{12}}) = \Span\{e_4, e_5, e_6\}$. So $k \leq 3$. We obtain $$Q_{\varphi_{12}} = \frac{1}{24}\Diag(-1, -1, -4, 2, -1, -1, 1)$$ with respect to $(e_i)_i$. Since $f$ is constant on $N$, $\hess f \big|_N = 0$. The gradient Laplacian soliton equation becomes $q_{\varphi_{12}}\big|_N = -\frac{1}{3}\lambda_{12}g_{\varphi_{12}}\big|_N$. But this implies $-1 = q_{\varphi_{12}}\big|_N(e_1, e_1) = q_{\varphi_{12}}\big|_N(e_3, e_3) = -4$, a contradiction.  
\end{proof}

\subsection{Observations on products}

We collect some immediate observations from the soliton equation in the product case, i.e., the case when $\Div \tau^2 = 0$. Recall that for products, $(T_{(p, q)}(N^{7 - k} \times \mathbb R^k), g) = (T_p N^{7 - k} \oplus T_q\mathbb R^k, g = g_N + g_{\mathbb R^k})$.

\begin{proposition}
If $(\varphi, \nabla f, \lambda)$ is a homogeneous closed gradient Laplacian soliton with $\tau^2$ divergence-free, i.e., $M = N^{7 - k} \times \mathbb R^k$, and $f$ is a function only on $\mathbb R^k$, we get the following:
\begin{enumerate}
    \item $0 = (-\ric_{g_N} - \frac{1}{2}\tau^2 + \frac{1}{3}(\scal_\varphi - \lambda) g)(X_i, X_j)$ for any $X_i, X_j \in T_pN$.

    \item $\tau^2(X, Y) = 0$ for any $X \in T_pN$ and $Y \in T_q\mathbb R^k$.

    \item $\hess f(Y_i, Y_j) = - \frac{1}{2}\tau^2(Y_i, Y_j) + \frac{1}{3}(\scal_\varphi - \lambda)g(Y_i, Y_j)$ for any $Y_i, Y_j \in T_q\mathbb R^k$. If in addition $\tau^2$ is a multiple of the metric, then $f$ is a Gaussian.

    \item $2^{-1}D_X\|\nabla f\|^2 = cg(\nabla f, X)$ where $c = 3^{-1}(\scal_\varphi - \lambda)$ is constant. Hence $f$ is an isoperimetric function as $\|\nabla f\|^2 = \phi(f).$ 
\end{enumerate}
\end{proposition}

We make some further observations when $f$ is a Gaussian.

\begin{corollary}
Suppose $(\varphi, \nabla f, \lambda)$ is a homogeneous closed gradient Laplacian soliton with $\tau^2$ divergence-free, i.e., $M = N^{7 - k} \times \mathbb R^k$, and $f$ is a Gaussian. Then 
\begin{enumerate}
    \item $\hess f(Y_i, Y_j) = cg(Y_i, Y_j)$ where constant $c = 3^{-1}(\scal_\varphi - \lambda)$ for any $Y_i, Y_j \in T_q\mathbb R^k$.
    \item $\tau^2(Y_i, Y_j) = 0$ for any $Y_i, Y_j \in T_q\mathbb R^k$.
    \item $\tau^2(X_i, X_j) = (-2\ric_{g_N} + 2cg)(X_i, X_j)$ for any $X_i, X_j \in T_pN$ and $$\tau^2 = \begin{bmatrix}
    -2\ric_{g_N} + 2cI_{7 - k} & \\ & 0_{k \times k}\end{bmatrix}$$ with respect to basis $(X_1, ... , X_{7 - k}, Y_1, ... , Y_k)$.
    \item $$\lambda = -\left(\frac{2 + k}{7 - k}\right)\scal_{g_N} =-\left(\frac{2 + k}{7 - k}\right)\scal_\varphi.$$
\end{enumerate} Hence the gradient soliton is either steady or expanding; $N^{7-k}$ must have constant nonpositive scalar curvature; and if $\varphi$ is closed non-torsion-free, then $N^{7 - k}$ must have constant negative scalar curvature. It also follows from (3) and (4) that $\tau^2$ is determined by $\dim N$ and $g_N$. 
\end{corollary}

\begin{proof} Since $f$ is a Gaussian on the Euclidean factor only, we have $\hess f(Y_i, Y_j) = cg(Y_i, Y_j)$ for some constant $c$ and $Y_i, Y_j \in T_q\mathbb R^k$. Setting $Y_i = Y_j = \nabla f$ and noting that $\nabla f \lrcorner \tau = 0$, by (3) of the preceding proposition, we get $c\|\nabla f\|^2 = 3^{-1}(\scal_\varphi - \lambda)\|\nabla f\|^2$. It follows that $c = 3^{-1}(\scal_\varphi - \lambda)$ since $\|\nabla f\| > 0$ as $f$ is assumed to be non-constant. Using (3) again yields $\tau^2(Y_i, Y_j) = 0$ for any $Y_i, Y_j \in T_q\mathbb R^k$. Furthermore, substituting $c$ in (1) of the preceding proposition, we get $$\tau^2(X_i, X_j) = (-2\ric_{g_N} + 2cg)(X_i, X_j) \, \, \, \, \,  \forall \, \, X_i, X_j \in T_pN.$$ Thus $\tau^2$ has the matrix representation as in (3) with respect to the basis $(X_1, ... , X_{7 - k}, Y_1, ... , Y_k)$.

Taking the trace yields $$\tr \tau^2 = -2\scal_{g_N} + \frac{2}{3}(7-k)(\scal_\varphi + \lambda).$$ Recall $-\frac{1}{2}\tr \tau^2 = -2\scal_\varphi$ and so $\tr \tau^2 = 4\scal_\varphi$. Putting this together with $\scal_\varphi = \scal_{g_N} + \scal_{g_{\mathbb R^k}} = \scal_{g_N}$ gives $$4\scal_{g_N} = -2\scal_{g_N} + \frac{2}{3}(7 - k)(\scal_{g_N} - \lambda),$$ from which we get $\lambda = -\left(\frac{2 + k}{7 - k}\right)\scal_{g_N} =-\left(\frac{2 + k}{7 - k}\right)\scal_\varphi.$ Since $\scal_\varphi \leq 0$ for closed $G_2$-structures, it follows that $\lambda \geq 0$, i.e., the soliton is either steady or expanding. From the expression for $\lambda$, the fact that $\scal_\varphi \leq 0$ also shows $\scal_N \leq 0$.
\end{proof}

\begin{remark}
The $0_{k \times k}$ block of $\tau^2$ in Corollary 3.5 (c) may be nonzero when $f$ is not Gaussian.
\end{remark}

The question arises of whether it is possible for $\tau^2$ to be a constant multiple of the metric. If so, then it would follow from Proposition 3.5 that $f$ is a Gaussian on $\mathbb R^k$ and that $g_N$ is an Einstein metric. But if $\tau^2 = cg$, for some nonzero $c \in \mathbb R$, then $\tau^2(\nabla f, \nabla f) = c\|\nabla f\|^2$. Since $\tau(\nabla f, \nabla f) = 0$ by Lemma 2.18 (4), it would follow that $\nabla f = 0$, a contradiction as we are considering non-trivial gradient solitons. Thus $\tau^2$ cannot be a constant multiple of the metric.

An open question remains whether there are any homogeneous closed gradient Laplacian solitons on products other than the Gaussian. If such non-trivial examples do exist, it would be desirable to obtain a classification of homogeneous closed gradient solitons on products. A more fundamental question arises of whether there  are homogeneous closed $G_2$-structures on product metrics $N^{7 -k} \times \mathbb R^k$. [There are known examples outside the homogeneous setting (see \cite{HN21, HKP22}).] We investigate this question for our choice of model (fundamental) $3$-form $\varphi$ from the introduction. The main observation is that to find closed $G_2$-structures on product metrics $N \times \mathbb R^k$, one should consider $\dim N \geq 4$.\\

\noindent \textit{Case $N^1 \times \mathbb R^6$}: We assume $(e_i)_{i = 1}^7$ is an basis such that $\{e_1\}$ is the basis for $T_pN^1$ and such that the $3$-form is the model form $\varphi =  e^{127} + e^{347} + e^{567} + e^{135} - e^{146} - e^{236} - e^{245}$ with respect $(e_i)_i$. Note that on product $N^1 \times \mathbb R^6$, the structure is given by $de^i = 0$ for all $i$. It is easy to see that $d\varphi = 0$ and so $N^1 = S^1$ or $\mathbb R$, i.e., the space is flat.\\

\noindent \textit{Case $N^2 \times \mathbb R^5$}: Let $(e_i)_{i = 1}^7$ be a basis for $T_{(p, q)}(N^2 \times \mathbb R^5)$ where $\{e_1, e_2\}$ and $\{e_3, e_4, e_5, e_6, e_7\}$ are bases for $T_pN^2$ and $T_q\mathbb R^5$, respectively. We have $\{e^{12}\}$ is a basis for $\Lambda^2(T_p^*N^2)$ and that the structure is given by $$de^1 = ae^{12}, \, \, \, \, \, de^2 = be^{12}, \, \, a, b \in \mathbb R  \, \, \, \, \, \text{and} \, \, \, \, \, de^i = 0 \, \, \forall \, \, i \ne 1, 2.$$ Suppose the model $3$-form $\varphi$ is with respect to the basis $(e_i)_i$. Then $d\varphi = ae^{1234} - ae^{1247} - be^{1236} - be^{1245} = 0$ if and only if $a = b = 0$. Thus in order for $\varphi$ to be closed in this basis, the space must be flat.\\

\noindent \textit{Case $N^3 \times \mathbb R^4$}: Let $(e_i)_{i = 1}^7$ be a basis for $T_{(p, q)}(N^3 \times \mathbb R^k)$ where $\{e_1, e_2, e_3\}$ and $\{e_4, e_5, e_6, e_7\}$ are bases for $T_pN^3$ and $T_q\mathbb R^4$, respectively. We have $\{e^{12}, e^{13}, e^{23}\}$ is a basis for $\Lambda^3(T_p^*N^3)$ and that the structure is given by $$\begin{cases}de^1 = a_{11}e^{12} + a_{12}e^{13} + a_{13}e^{23} \\ de^2 = a_{21}e^{12} + a_{22}e^{13} + a_{23}e^{23} \\ de^3 = a_{31}e^{12} + a_{32}e^{13} + a_{33}e^{23} \\ de^i = 0 \,\,\,\,\, \forall \,\, i \ne 1, 2, 3.\end{cases}$$ By straightforward computations, $d\varphi = 0$ if and only if $a_{ij} = 0$ for all $i, j = 1, 2, 3$. We obtain again that in order for $\varphi$ to be closed, the space must be flat.\\

\noindent \textit{Case $N^4 \times \mathbb R^3$}: By similar setup and computations as in the preceding cases, we get that $d\varphi = 0$ yields an undetermined system of $12$ equations in $24$ unknowns. We do not know if $N^4 \times \mathbb R^3$ can admit closed $G_2$-structures.\\

\noindent \textit{Case $N^{5} \times \mathbb R^2$}: 
A non-trivial example of a homogeneous product of the form $N^5 \times \mathbb R^2$ admitting a closed $G_2$-structure is the space $K^7 = H(1, 2) \times \mathbb R^2$ constructed in \cite{Fer87} where $$H(1, 2) = \left\{\begin{bmatrix} I_2 & X & Z \\
0 & 1 & y \\ 0 & 0 & 1\end{bmatrix} \Bigg| X = (x_1, x_2)^t, Z = (z_1, z_2)^t, x_i, z_j, y \in \mathbb R \right\}$$ is the generalized Heisenberg group. It is known that $K^7$ is a connected nilpotent Lie group. We show there is a Lie algebra isomorphism taking the dual basis $(f_j)_j$ to the basis $(e_i)_i$ for $(\mathfrak n_2, \varphi_2)$ in \cite{Nic18}. We label the left-invariant $1$-forms on $K^7$: $$f^1 = dx_1, f^1 = dx_2, f^3 = dy, f^4 = dz_1 - x_1dy, f^5 = dz_2 - x_2dy, f^6 = du_1, f^7 = du_2.$$ The structure on $K^7$ is $$df^4 = -f^{13}, \, \, \, \, \, df^5 = -f^{23}, \, \, \, \, \,\text{and} \, \,  df^j = 0 \, \, \forall \, \, j \ne 4, 5,$$ or equivalently, $[f_1, f_3] = f_4, [f_2, f_3] = f_5,$ and $[f_s, f_t] = 0$ for all other $s, t$. The metric is given by $\sum_j(f^j)^2$. Let $(e_i)_i$ be the basis for $(\mathfrak n_2, \varphi_2)$ which has structure $[e_1, e_2] = -e_5$ and $[e_1, e_3] = -e_6$. Then the Lie algebra isomorphism $h: (\mathfrak n_2, \varphi_2) \rightarrow (K^7, \varphi_{K^7})$ taking $$e_1 \mapsto f_3, e_2 \mapsto f_1, e_3 \mapsto f_2, e_4 \mapsto f_7, e_5 \mapsto f_4, e_6 \mapsto f_5, e_7 \mapsto f_6$$ satisfies $h \cdot \varphi_2 = \varphi_{K^7}$ where $\varphi_{K^7} = -f^{147} + f^{257} + f^{156} + f^{246} + f^{345} + f^{123} - f^{367}$ is the closed $G_2$ structure on $K^7$. We do not know whether $K^7$ admits gradient Laplacian solitons.\\

\noindent \textit{Case $N^6 \times \mathbb R$}: This is a special case of the construction of one-dimensional extensions discussed in the next section (see Remark 4.6 in Section~\ref{sec:4}).

\section{Gradient Laplacian Solitons on Almost Abelian Solvmanifolds}\label{sec:4}

A $G$-homogeneous space $(M = G/G_x, g)$ is a \textit{one-dimensional extension} if there is a closed subgroup $H \subset G$ contaning $G_x$ such that there is a surjective Lie group homomorphism $G \rightarrow (\mathbb R, +)$ with kernel $H$. In this section, we study one-dimensional extensions admitting closed $G_2$-structures with the main goal of proving Theorem 1.8. We first recall the setup for one-dimensional extensions in more detail.

\subsection{Setup for one-dimensional extensions}

Let $H$ be a Lie group and $(M = H/K, g)$ an $H$-homogeneous space. Let $\mathfrak h$ and $ \mathfrak k$ be the Lie algebras of $H$ and $K$, respectively. The family of automorphisms $\{\Phi_t\}_t \subset \Aut(H)$ such that $\Phi_t(K) = K$ induces a well defined family of diffeomorphisms $\{\phi_t\}_t \subset \Diff(H/K)$ given by $$\phi_t(hK) = \Phi_t(h)K \, \, \, \, \, \forall\, \, h \in H.$$ We fix an $\Ad(K)$-invariant decomposition $\mathfrak h = \mathfrak p \oplus \mathfrak k$. We can identify $\mathfrak p \equiv T_xM$ via the orthogonal projection $\mathfrak h \rightarrow \mathfrak p$.

Now suppose $H$ is a Lie group with $(N = H/K, h)$ a $H$-homogeneous space. Fix a derivation $D \in \text{Der}(\mathfrak h)$ that preserves $K$, an isotropy subgroup at some point $x \in N$. To obtain a one-dimensional extension of $(N, h)$, we consider the Lie aglebra $$\mathfrak g = \mathfrak h \oplus_D \mathbb R\xi$$ with Lie bracket given by $$\ad_\xi(X) = D(X) \, \, \, \, \, \text{and} \, \, \, \, \, \ad_Y(X) = \ad_Y^{\mathfrak h}(X) \, \, \, \, \,   \forall \, \, X, Y \in \mathfrak h.$$ Let $G$ be the simply-connected Lie group with Lie algebra $\mathfrak g$. Then
\begin{enumerate}[(i)]
    \item $G \supset H$, a codimension one normal subgroup of $G$ as $\ad_\xi(X) \in \mathfrak h$ for all $X \in \mathfrak h$;
    \item $G = H \ltimes \mathbb R$;
\item any $\Ad(K)$-invariant decomposition $\mathfrak h = \mathfrak p \oplus \mathfrak k$ yields a corresponding $\Ad(K)$-invariant decomposition $\mathfrak g = \mathfrak q \oplus \mathfrak k = (\mathfrak p \oplus \mathbb R\xi) \oplus \mathfrak k$;
\item $G$-invariant metrics are identified with restrictions of $\Ad(K)$-invariant inner products on $\mathfrak g$ to $\mathfrak q$. 
\end{enumerate}  The $G$-homoegeneous space $(M = G/K, g)$ where the metric satisfies $g\big|_{\mathfrak p} = h$, $g(\xi, X) = 0$ for all $X \in \mathfrak p$, and $g(\xi, \xi) = 1$ is the one-dimensional extension of $(N, h)$. The one-dimensional extensions obtained in this way are equivalent to the ones described at beginning of this section.

Connections between Ricci solitons and Einstein metrics on such homogeneous spaces have been studied by He-Petersen-Wylie in \cite{HPW15}. We will need [Lemma 2.9, \cite{HPW15}].

\begin{lemma}[Lemma 2.9, \cite{HPW15}] The Ricci tensor of one-dimensional extensions $(M, g)$ with Lie algebra of the form $\mathfrak g = \mathfrak h \oplus_D \mathbb R \xi$ is given by 

\begin{enumerate}
    \item $\ric(\xi, \xi) = -\tr(S^2)$
    \item $\ric(X, \xi) = - \Div (S)$
    \item $\ric(X, X) = \ric^N(X, X) - (\tr S)h(S(X), X) - h([S, A](X), X)$,
\end{enumerate} where $S = (D + D^t)/2$ and $A = (D - D^t)/2$, the symmetric and skew-symmetric parts of $D$, respectively. 

\end{lemma}

Let $\mathfrak g = \mathfrak h \oplus_D \mathbb Re_7$ be the Lie algebra of Lie group $M$. It is known that if $\mathfrak h$ has an $\SU(3)$-structure, then there is an orthonormal basis $\{e_1, ... , e_6\}$ for $\mathfrak h$ such that the $\SU(3)$-structure is characterized by the pair of forms $(\omega, \rho^+) \in \Lambda^2\mathfrak h^* \times \Lambda^3\mathfrak h^*$ where $$\omega = e^{12} + e^{34} + e^{56} \, \, \, \, \, \text{and} \, \, \, \, \, \rho^+ = e^{135} - e^{146} - e^{236} - e^{245}.$$ Then $\varphi = \omega \wedge e^7 + \rho^+ = e^{127} + e^{347} + e^{567} + e^{135} - e^{146} - e^{236} - e^{245}$ is a $G_2$-structure on $M$. We note that $\rho^- = -e^{246} + e^{235} + e^{145} + e^{136}$ is the complex part of the complex volume $(3, 0)$-form $\Psi = \rho^+ + i\rho^-$ from the $\SU(3)$-structure $(g, J, \Psi)$ on $\mathfrak h$. If both forms $(\omega, \rho^+)$ are closed, we say that the $\SU(3)$-structure is \textit{symplectic half-flat}. It is clear that if $\mathfrak h$ is symplectic half-flat, then the $G_2$-structure $\varphi = \omega \wedge e^7 + \rho^+$ is closed. Manero showed in \cite{Man15} that $\mathfrak h \subset \mathfrak g$ admitting symplectic half-flat $\SU(3)$-structure is equivalent to $\varphi = \omega \wedge e^7 + \rho^+$ being closed  whenever $D$ is the real representation of some $A \in \mathfrak {sl}(3, \mathbb C)$. [Manero uses the classification of symplectic half-flat $\SU(3)$-structures on solvable Lie algebra $\mathfrak h$ to construct new examples of closed $G_2$-structures in \cite{Man15}].

If in addition $\mathfrak h$ is an abelian ideal, we call $\mathfrak g$ \textit{almost abelian} and $M$ an \textit{almost abelian solvmanifold}. The Lie algebras for almost abelian solvmanifolds are completely determined by derivation $D: \mathfrak h \rightarrow \mathfrak h$ defined by $$D(e_i) = \ad_{e_7}\big|_{\mathfrak h}(e_i) = [e_7, e_i \big|_{\mathfrak h}].$$ [Note: $D$ coincides with $A$ in \cite{Lau17a}.] For almost abelian solvmanifolds, $\varphi = \omega \wedge e^7 + \rho^+$ is closed if and only if the derivation $D$ is the real representation of some element $A \in \mathfrak {sl}(3, \mathbb C)$ (see \cites{Fre13, Lau17a}).

\noindent \textit{Notation:} We write $(G_D, g)$ to denote the Lie group with Lie algebra $\mathfrak g = \mathfrak h \oplus_D \mathbb R\xi$ and $\mu_D = [\cdot, \cdot]_D$ to denote the Lie bracket of $\mathfrak g$ obtained from $D$. We write $\ric_D$ and $\scal_D$ for the Ricci and scalar curvatures from the metric $g$, respectively, where $g$ is the extension of the metric $h$ on $N$. We also write $\tau_D$ and $Q_D$ for the torsion form and unique symmetric operator from \ref{eq:2.4} corresponding to $(G_D, \varphi)$ when $\varphi$ is closed.

We now collect some facts regarding $d_{\mathfrak g}:\Lambda^\ell \mathfrak g^* \rightarrow \Lambda^{\ell + 1} \mathfrak g^*$ and $d_{\mathfrak h}: \Lambda^k \mathfrak h^* \rightarrow \Lambda^{k + 1} \mathfrak h^*$, the exterior derivatives (\textit{Chevalley-Eilenberg differentials}) on $\mathfrak g$ and $\mathfrak h$, respectively, which will be used in later computations. These facts are included or deduced from [Lemma 5.12, \cite{Lau17a}]. It is known that if $\omega$ is an invariant $k$-form on a Lie group, then $\omega(X_1, ... , X_k)$ is a constant. In particular, if $\gamma$ is a $1$-form, then by [Proposition 2.19, \cite{Lee13}] $$d\gamma(X, Y) = -\gamma([X, Y]).$$ Let $(e_i)_i$ be a basis of left-invariant vector fields and $(e^i)_i$ be its co-basis of left-invariant $1$-forms. Since $e^i$ is an invariant $1$-form, $$de^i(e_j ,e_k) = -e^i([e_j, e_k]).$$ Then the structure equations are $$[e_j, e_k] = c_{jk}^\ell e_\ell,$$ where $c_{jk}^\ell$ are the structure constants. It follows that $$de^i = -c_{jk}^ie^{jk}.$$ Also recall the map $\theta: \mathfrak {gl}(\mathfrak h)\rightarrow \text{End}(\Lambda^k \mathfrak h^*)$ is the representation obtained as the derivative of the natural $\GL(\mathfrak h)$-action $h \cdot \gamma = (h^{-1})^*\gamma$: $$\theta(D)\gamma  = -\gamma(D \cdot, ..., \cdot) - \cdots - \gamma(\cdot, ..., D\cdot) \, \, \, \, \, \forall \, \, \gamma \in \Lambda^k\mathfrak h^*.$$

\begin{lemma}
    Given Lie algebra $\mathfrak g = \mathfrak h \oplus_D \mathbb R e_7$, where derivation $D: \mathfrak h \rightarrow \mathfrak h$ defined by $D(X) = [e_7, X]$ for all $X \in \mathfrak h$ determines the structure equations, the following holds:

    \begin{enumerate}[(i)]
        \item $d_{\mathfrak g}e^7 = 0$.
    
        \item $d_{\mathfrak g}\gamma = d_{\mathfrak h}\gamma + (-1)^k(\theta(D)(\gamma))\wedge e^7$ for any $\gamma \in \Lambda^k \mathfrak h^*$.

        \item $d_{\mathfrak g}(\gamma \wedge e^7) = d_{\mathfrak h} \gamma \wedge e^7$ for any $\gamma \in \Lambda^k \mathfrak h^*$. 

        \item For $\varphi = \omega \wedge e^7 + \rho^+$, $$d_{\mathfrak g} \varphi = d_{\mathfrak h}\omega \wedge e^7 + (d_{\mathfrak h}\rho^+ -(\theta(D)\rho^+) \wedge e^7).$$ Thus $\varphi$ is closed if and only if $$d_{\mathfrak h}\omega = 0, d_{\mathfrak h} \rho^+ = 0, \, \, \text{and} \, \, \theta(D)\rho^+ = 0.$$

        \item $\theta(D)\rho^+ = 0$ if and only if $\theta(D)\rho^- = 0$, if and only if $D \in \mathfrak {sl}(3, \mathbb C)$.

        \item If $\tr(D) = 0$, then $\theta(D)*_{\mathfrak h} = -*_{\mathfrak h} \theta(D^t)$ on $\Lambda \mathfrak h^*$.
    \end{enumerate}
\end{lemma}

\begin{proof}
Statement (i) follows from the fact that $[e_i, e_j] \in \mathfrak h$ for all $i, j$; (ii) follows from [Proposition 2.19, \cite{Lee13}] and the fact that $\mathfrak h$ is not assumed to be abelian, hence the term $d_{\mathfrak h}\gamma$ appears; (iii) follows from (ii); (iv) follows from (i)-(iii). Statements (v) and (vi) are statements of Lemma 5.12 (iv) and (v) in \cite{Lau17a}.  
\end{proof}

\begin{remark} Note the second term in (ii) is $d_A\gamma$ in [Lemma 5.12 (ii), \cite{Lau17a}].
\end{remark}

\subsection{Proof of Theorem 1.8}

\begin{theorem}
If $(\varphi, \nabla f, \lambda)$ is a closed non-torsion-free gradient Laplacian soliton on Lie group $G_D$ with Lie algebra $\mathfrak g = \mathfrak h \oplus_D \mathbb Re_7$ and $\mathfrak h$ is a codimension-one abelian ideal, then it must be a product $N \times \mathbb R^k$ and $f$ is constant on $N$.
\end{theorem}

\begin{proof}[Proof Outline]
Suppose $G_D$ admits a gradient Laplacian soliton $(\varphi, \nabla f, \lambda)$. Then by the Structure Theorem, either it is a one-dimensional extension or it is product. If $G_D$ is a one-dimensional extension, then the potential function is either of the form $f = ar + b$ or $f(x,y) = ar(x) + v(y)$ and either $\nabla r = \pm e_7$ or $\nabla r \ne \pm e_7$. If $\nabla r = \pm e_7$, then by Proposition 4.9, the space is flat, contradicting $\varphi$ being closed non-torsion-free. If $\nabla r \ne \pm e_7$, then by Proposition 4.11, the space is also flat, contradicting $\varphi$ being closed non-torsion-free. Thus by the Structure Theorem, $G_D$ must be a product $N \times \mathbb R^k$ with $f$ constant on $N$.
\end{proof}

The rest of this section consists of observations culminating in the propositions used to prove Theorem 1.8. We fix some notation for the statements to follow. We consider only Lie algebras $\mathfrak g = \mathfrak h \oplus_D \mathbb Re_7$ where the derivation $D: \mathfrak h \rightarrow \mathfrak h$ given by $D(X) = [e_7, X]$ for all $X \in \mathfrak h$ is the real representation of some $A \in \mathfrak {sl}(3,\mathbb C)$. Let $S$ be the symmetrization of $D$, i.e., $S = (D + D^t)/2$. The hypothesis that the simply-connected Lie group $(G_D, \varphi)$ is a closed $G_2$-structure in the following statements can be replaced by $(\mathfrak h, \omega, \rho^+)$ being a symplectic half-flat $\SU(3)$-structure by the result of Manero. We first obtain general matrix formulas for the operators in \ref{eq:2.4} in the case of one-dimensional extensions. These matrix formulas generalize matrix formulas in the almost abelian case found in \cite{Lau17a} to the not almost abelian case.

\begin{proposition}
 Suppose $G_D$ has Lie algebra of the form $\mathfrak g = \mathfrak h \oplus_D \mathbb R e_7$ and admits closed $G_2$-structure $\varphi$. Then with respect to an orthornomal basis $(e_i)_{i = 1}^7$ where $\mathfrak h = \Span\{e_1, ... , e_6\}$, we have the following:
\begin{enumerate}
    \item $\tau_D^2 = \begin{bmatrix} -(D + D^t)^2 + J(D + D^t)B + BJ(D + D^t) + B^2 & \\ & 0\end{bmatrix}$ where $B = *_{\mathfrak h}d_{\mathfrak h} \rho^-$.
    \item The matrix representation for $\ric_D$ is $$\ric_D = \begin{bmatrix} \ric^H & \\ & 0\end{bmatrix}  + \begin{bmatrix} \frac{1}{2}[D, D^t] & \\ & -\frac{1}{4}\tr((D + D^t)^2) \end{bmatrix} - P,$$ where $P = \begin{bmatrix} 0 & \cdots & 0 & \Div(S)(e_1) \\ \vdots & \ddots & \vdots & \vdots\\ 0 & \cdots & 0 & \Div(S)(e_6) \\ \Div (S)(e_1) & \cdots & \Div (S)(e_6) & 0 \end{bmatrix}$.

    \item $\scal_D = \scal^H - \tr(S^2)$.

    \item $Q_D = \begin{bmatrix}Q_H & \\ & q_{e_7} \end{bmatrix} - P$ where  

\begin{align*}
    Q_H = &\ric_H + \frac{1}{2}[D, D^t] -\frac{1}{4}\tr(D + D^t)(D + D^t) \\&+\frac{1}{2}[-(D + D^t)^2 + J(D + D^t)B + BJ(D + D^t) + B^2] \\&-\frac{1}{3}[\scal_H - (\tr S)^2  -\tr(S^2)]I_{6 \times 6},
\end{align*} and $q_{e_7}  = -\frac{1}{3}\scal_H + \frac{1}{3}(\tr S)^2 -\frac{2}{3}\tr(S^2).$
\end{enumerate}
\end{proposition}

\begin{proof}
Recall there exists an orthonormal basis $(e_i)_{i = 1}^7$ such that $\mathfrak h = \Span (e_i)_{i = 1}^6$ and $\varphi = \omega \wedge e^7 + \rho^+$. Since the structure is determined by $D$, we write $\tau_{\mathfrak g} = \tau_D$. We compute the intrinsic torsion $\tau_{\mathfrak g} = - * d_{\mathfrak g} *\varphi$ for closed $G_2$-structure $\varphi$ using linear algebra properties of $*$ on $\Lambda^k \mathfrak g^*$ and $*_{\mathfrak h}$ on $\Lambda^k \mathfrak h^*$ from Lemma 5.11 of \cite{Lau17a}. Taking the Hodge star of $\varphi$ we get $$ *\varphi= *(\omega \wedge e^7) + *\rho^+ = (-1)^2*_{\mathfrak h}\omega + *_{\mathfrak h}\rho^+ \wedge e^7 \\ = \dfrac{1}{2}\omega \wedge \omega + \rho^- \wedge e^7,$$ where the first equality follows from Lemma 5.11 (i) and (ii), while the second equality follows from Lemma 5.11 (iii) and (iv).  Taking the differential yields
    $$d_{\mathfrak g}(*\varphi) = d_{\mathfrak g}(\dfrac{1}{2}\omega \wedge \omega) + d_{\mathfrak g}(\rho^- \wedge e^7) = \dfrac{1}{2}d_{\mathfrak g}(\omega \wedge \omega) + d_{\mathfrak h}\rho^- \wedge e^7$$ as $\rho^- \in \Lambda^3 \mathfrak h^*$. The first term in the last expression is \begin{align*}
    \frac{1}{2}d_{\mathfrak g}(\omega \wedge \omega) &= \frac{1}{2}[d_{\mathfrak h}(\omega \wedge \omega) + (-1)^{4 + 1}(\theta(D)(\omega \wedge \omega)) \wedge e^7] \\ &= \frac{1}{2}[(d_{\mathfrak h}\omega \wedge \omega + (-1)^2\omega \wedge d_{\mathfrak h}\omega) - (\theta(D)(\omega \wedge \omega)) \wedge e^7] \\ &= -\frac{1}{2}(\theta(D)(\omega \wedge \omega))\wedge e^7 \\ &= -  \theta(D)(\dfrac{1}{2}\omega \wedge \omega) \wedge e^7 \\ &= -\theta(D)*_{\mathfrak h} \omega \wedge e^7 \\ &= *_{\mathfrak h}\theta(D^t)\omega \wedge e^7,
\end{align*} where we used $d_{\mathfrak h}\omega = 0$ in the third equality, Lemma 5.11 (iii) in the fifth, and Lemma 5.12 (vi) in the last as $\tr D = 0$. If $d_{\mathfrak h}\rho^- \wedge e^7 \ne 0$, we get $$d*\varphi =  -*_{\mathfrak h}\theta(D^t)\omega \wedge e^7 + d_{\mathfrak h}\rho^- \wedge e^7.$$ Taking the Hodge star again gives

\begin{align*}*d_{\mathfrak g}*\varphi &= *(-*_{\mathfrak h}\theta(D^t)\omega \wedge e^7 + d_{\mathfrak h}\rho^- \wedge e^7) \\ &= (-1)^4*_{\mathfrak h}(-*_{\mathfrak h} \theta(D^t)\omega) + (-1)^4*_{\mathfrak h}d_{\mathfrak h}\rho^- \\ &= -*_{\mathfrak h}^2\theta(D^t)\omega  + *_{\mathfrak h}d_{\mathfrak h}\rho^- \\ &= -(-1)^2\theta(D^t)\omega + *_{\mathfrak h}d_{\mathfrak h}\rho^-\\ &= -\theta(D^t)\omega + *_{\mathfrak h}d_{\mathfrak h}\rho^-.\end{align*} where we used Lemma 5.11 (ii) for the second equality and Lemma 5.11 (v) for the second to last equality. Then $$\tau_D = - *d_{\mathfrak g}* \varphi = \theta(D^t)\omega - *_{\mathfrak h}d_{\mathfrak h}\rho^-.$$ Note, $\rho^- \in \Lambda^3 \mathfrak h^*$ implies $d_{\mathfrak h}\rho^- \in \Lambda^4 \mathfrak h^*$. Taking the Hodge star yields $*_{\mathfrak h} d_{\mathfrak h} \rho^- \in \Lambda^2 \mathfrak h^*$, i.e., $*_{\mathfrak h} d_{\mathfrak h} \rho^-$ is a $2$-form on $\mathfrak h^*$ and thus can be written as a matrix with respect to $2$-forms $(e^{ij})_{i, j}$; $i, j = 1, ..., 6$. Set $B:= *_{\mathfrak h} d_{\mathfrak h} \rho^-$. Then the matrix representation of the torsion $2$-form is $$\tau_D = \begin{bmatrix} -J(D + D^t)& \\ & 0\end{bmatrix} - \begin{bmatrix} B &  \\ & 0\end{bmatrix}.$$ Taking the square of $\tau_D$ yields the matrix representation of $\tau_D^2$.

We use Lemma 4.1 to obtain $\ric_D$. Let $g_H$ denote the metric on $H$ corresponding to $\mathfrak h$. By Lemma 4.1 (1) with $\xi = e_7$ and $\alpha = 1$, we have $$\ric_D(e_7, e_7) = -\tr(S^2).$$ Also, Lemma 4.1 (2) and symmetry of $\ric_D$ yields

$$\ric_D(e_7, e_i) = \ric_D(e_i, e_7) = -\Div(S)(e_i) \, \, \, \, \, \forall \, \,  i = 1, ... , 6.$$ Note $-[S, A] = [A,S] = \frac{1}{2}[D, D^t].$ Then for $i, j = 1, ..., 6$, Lemma 4.1 (3) gives

\begin{align*}\ric_D(e_i, e_j) = &\ric^H(e_i, e_j) \\ &- \frac{1}{4}\tr(D + D^t)g_H((D + D^t)(e_i), e_j) + g_H(\frac{1}{2}[D, D^t](e_i), e_j).\end{align*} Putting these observations together and using the fact that $\tr D = \tr D^t = 0$ yields the matrix representation of $\ric_D$. 

The expression for $\scal_D$ follows from taking the trace of $\ric_D$ and the fact that $\tr[D, D^t] = 0$. The matrix formula for $Q_D$ follows from [Proposition 2.2, \cite{Lau17a}] and the preceding results.
\end{proof}

\begin{remark}
    A one-dimensional extension is a product metric if and only if the derivation $D$ is anti-symmetric. By the discussion preceding Proposition 4.5, a product metric $N^6 \times \mathbb R$ has a closed $G_2$-structure if $N^6$ admits an anti-symmetric derivation and a symplectic half-flat $\SU(3)$-structure. We do not know of any examples of symplectic half-flat $\SU(3)$ structures that admit an anti-symmetric derivation. Moreover, in order for such a metric to be a closed gradient Laplacian soliton, by Proposition 4.5, we see that $\ric^N - \frac{1}{2}B^2$ must be a constant multiple of the metric $g_N$.
\end{remark}

\subsection{Case: $\nabla r = e_7$}

\begin{proposition}
 Suppose $G_D$ has Lie algebra of the form $\mathfrak g = \mathfrak h \oplus_D \mathbb R e_7$ and admits closed $G_2$-structure $\varphi$. If $$(\varphi = \omega \wedge e^7 + \rho^+, \nabla f, \lambda)$$ is a closed gradient Laplacian soliton where $f(r) = ar + b$ and $\nabla r= e_7$, then 
\begin{enumerate}
    \item $\Div (S)(X) = 0 \, \, \, \, \, \forall \, \, X \in \mathfrak h.$ 
    \item $\Div(S)(\nabla r) = \Div(S)(e_7) = \tr(S^2).$
    \item $\lambda = \scal^H + 2\tr(S^2)$.
    \item $\Delta f = - 2\tr(JSB) - \frac{1}{2}\tr B^2 - 4\tr(S^2).$
\end{enumerate}
\end{proposition}

\begin{proof} By Proposition 4.5 (2) we have
$$\ric_D = \begin{bmatrix} \ric^H & \\ & 0\end{bmatrix} + \begin{bmatrix} \frac{1}{2}[D, D^t] & \\ & -\frac{1}{4}\tr((D + D^t)^2) \end{bmatrix} - P,$$ where $P$ is the matrix with $0$ entries for all $(i, j)$ except for the $(i, 7), (7,i)$-entries, where it is $\Div(S)(e_i)$ for $i = 1, ... , 6$. Since $\hess f(\nabla r) = a\nabla_{\nabla r}\nabla r = 0$, the gradient Laplacian soliton equation applied to $\nabla r = e_7$ becomes \begin{equation}\label{eq:4.1} 0 = -\ric_D(e_7) - \frac{1}{2}\tau_D^2(e_7) + \frac{1}{3}(\scal_D - \lambda)I(e_7)\, \, \, \, \, \end{equation} By Lemma 4.1 (1) and equation (\ref{eq:4.1}), we get \begin{align*}-\Div(S)(e_i) &= \ric_D(e_7, e_i) \\ &= -\frac{1}{2}g(\underbrace{\tau_D^2(e_7)}_{= 0}, e_i) + \frac{1}{3}(\scal_D - \lambda)\underbrace{g(e_7, e_i)}_{= 0} = 0  \end{align*} for $i = 1, ..., 6$. Thus $$\Div(S)(X) = 0 \, \, \, \, \, \forall \, \, X \in \mathfrak h.$$

Recall that by [Proposition 2.7, \cite{HPW15}], the shape operator $T(X) = \nabla_X^{G_D} e_7$ is related to symmetrization $S$ by $T = -S$. So with $e_7 = \nabla r$, we have $$S(X) = -T(X) = -\nabla_X\nabla r.$$ Then $$\Div(S)(\nabla r) = -\ric_D(\nabla r, \nabla r) - D_{\nabla r}(\Delta r) = -\ric_D(\nabla r, \nabla r) = -\ric_D(e_7, e_7) = \tr(S^2),$$ where the first equality follows from a Bochner formula [Lemma 2.1, \cite{PW09}]; the second equality follows from $\Delta r$ being constant on one-dimensional extensions; and the last equality follows from Lemma 4.1 (1).

The expression for $\lambda$ is obtained from solving for $\lambda$ in equation (\ref{eq:4.1}) and using the expression for $\scal_D$ from Proposition 4.5 (3). Finally, $\Delta f$ is obtained from taking the trace of the soliton equation $\hess f = -Q_D - (1/3)\lambda I$ and substituting the expression in (3) for $\lambda$. 
\end{proof}

\begin{remark}
Equation (\ref{eq:4.1}) in the proof of Proposition 4.5 requires that the gradient soliton has potential function of the form $f = ar + b$ and that $\nabla r = e_7$. In particular, the hypothesis $\nabla r = e_7$ is needed to obtain the explicit expressions for $\lambda$ and $\Delta f$ in terms of $S$.  
\end{remark}

\begin{corollary}
If $\mathfrak h$ is an abelian ideal in addition to the hypotheses of Proposition 4.7, then $$\lambda = -2\scal_D.$$ That is, such a gradient soliton must be expanding and is steady if and only if $\varphi$ is torsion-free. Moreover, 
\begin{enumerate}[(i)]
    \item $\Delta f = 4\scal_D$;
    \item $-\frac{1}{2}\Div \tau_D^2(e_i) = \begin{cases} 0 & i = 1, ..., 6\\
    -\tr(S^2) = \scal_D &i = 7; \, \, \, \, \, ( e_7 = \nabla r)
    \end{cases}.$
\end{enumerate}
\end{corollary}

\begin{proof} In the setting of almost abelian solvmanifolds admitting closed $G_2$-structure $\varphi$, the terms $B, \tr D, \tr D^t, \tr S, \ric^H, \scal^H$ are all $0$. The formulas for $\lambda$ and $\Delta f$ immediately follow from these observations and Proposition 4.7. Since $\scal^H = 0$, $\scal_D = -\tr(S^2)$. By the Key Lemma and Lemma 4.1 (1) and Lemma 4.1 (2), we have $$-\frac{1}{2}\Div\tau_D^2(e_i) = \ric_D(\nabla r, e_i) = -\Div (S)(e_i),$$ which by Proposition 4.7 is $0$ for $i = 1, ... , 6$ and $-\tr(S^2)$ for $i = 7$.  
\end{proof}

\begin{theorem}
Let $G_D$ be a Lie group with Lie algebra of the form $\mathfrak g = \mathfrak h \oplus_D \mathbb R e_7$ where $\mathfrak h$ is a codimension-one abelian ideal. Suppose $G_D$ admits closed gradient Laplacian soliton $$(\varphi = \omega \wedge e^7 + \rho^+, \nabla f, \lambda)$$ where $G_D$ is a one-dimensional extension as in the Structure Theorem. In particular, suppose the potential function is either of the form $f = ar + b$ or $f(x, y) = ar(x) + v(y)$ with $\nabla r= e_7$. Then the soliton is steady, $\varphi$ is torsion-free, and $G_D$ is flat.
\end{theorem}

\begin{proof}
In the case of almost abelian solvmanifolds, by Proposition 4.5 we have $$\ric_D = \begin{bmatrix} \frac{1}{2}[D, D^t] & \\ & -\frac{1}{4}\tr(D + D^t)^2 \end{bmatrix} \, \, \, \, \, \text{and} \, \, \, \, \, \tau_D^2 = \begin{bmatrix} -(D + D^t)^2 & \\ & 0 \end{bmatrix}.$$ [These matrix expressions also follow from results of Lauret (see \cite{Arr13}, \cite{Lau11}, and \cite{Lau17a}).] Suppose $(\varphi, \nabla f, \lambda)$ is a gradient Laplacian soliton where the potential function is of the form $f = ar + b$ with $\nabla r = e_7$. Since $\mathfrak h$ is abelian, Corollary 4.9 says that $\lambda = - 2\scal_D$. Substituting this expression for $\lambda$ in the soliton equation gives

\begin{align*}
    \hess f = a\hess r &= -\ric_D - \frac{1}{2}\tau_D^2 + \scal_DI \\ &= -\begin{bmatrix} \frac{1}{2}[D,D^t] & \\ & \scal_D\end{bmatrix} - \frac{1}{2}\begin{bmatrix} -(D + D^t)^2 & \\ & 0 \end{bmatrix} + \scal_D I_{7 \times 7} 
\end{align*} Taking the trace yields $$\Delta f = - \scal_D + 2\tr(S^2) + 7\scal_D = 4\scal_D$$ where we used that $\tr[D, D^t] = 0$. By Proposition 4.5 (3), $\scal_D = -\tr(S^2)$, and so  $$\Delta f = -4\tr(S^2).$$ Recall the shape operator $T = \nabla_{\cdot}e_7 = -S$ and so $$-S(X) = T(X) = \nabla_{X}e_7 = \nabla_X\nabla r =  \frac{1}{a}\nabla_X \nabla f = \frac{1}{a}\hess f(X).$$ Hence $\hess f = -aS$ and taking the trace yields $\Delta f = -a\tr(S) = 0,$ where the last equality follows from $\tr(D) =  \tr(D^t) = 0$. Putting this together with the expression obtained for $\Delta f$ above gives $\tr(S^2) = 0.$ Thus $S = 0$ and by Proposition 4.5 (3) we get $\scal_D = 0$.

Note $S = 0$ if and only if $D = -D^t$, i.e., $D$ is antisymmetric. This together with $\tr(S^2) = 0$ gives $\ric_D = 0$, i.e., the space is Ricci-flat. By a result of Alekseevski\u{\i}-Kimel\cprime fel\cprime d in \cite{AK75}, Ricci-flat homogeneous spaces are flat, and so $(G_D, g_\varphi)$ is flat. Moreover, $\lambda = -2\scal_D = 0 $, i.e., the soliton is steady. Furthermore, since $\scal_D = 0$, $\varphi$ is torsion-free.

In the case where the potential function is of the form $f(x, y) = ar(x) + v(y)$, recall that the function $v$ on the Euclidean factor is in $\{\hess v = 0\}$. Then $\hess f = a\hess r$ and so we can run through the same arguments above to get that the space is flat, the soliton is steady, and $\varphi$ is torsion-free. 
\end{proof}

\begin{remark}
If we start with a flat space, we can choose potential function $f = ar + b$ such that the gradient Laplacian soliton equation is satisfied by taking $r$ to be the coordinate of one of the unit basis vectors, $\nabla r = e_i$, and $\lambda = 0$. One can also construct a Gaussian on flat space $\mathbb R^7$ (see, e.g., the case $\mathfrak n_1$ in Section~\ref{sec:3}).
\end{remark}

\subsection{Case: $\nabla r \ne \pm e_7$}

\begin{theorem} Let $D$ and $D'$  be derivations of $6$-dimensional subalgebras $\mathfrak h$ and $\mathfrak h'$ of a $7$-dimensional Lie algebra $(\mathfrak g, [\cdot, \cdot]_{\mathfrak g})$ defined by $D(X) = [e_7, X]$ and $D(Y) = [\nabla r, Y]$, respectively. Suppose $\mathfrak h$ is codimension-one abelian ideal. 
    Let $M$ be the Lie group corresponding to Lie algebra $\mathfrak g$ with decompositions $$ \mathfrak h \oplus_D \mathbb Re_7 = \mathfrak h' \oplus_{D'} \mathbb R\nabla r,$$ where $\nabla r \ne \pm e_7$. Suppose $(\varphi, \nabla f, \lambda)$ is a closed gradient Laplacian soliton and $(M, g_\varphi)$ is one-dimensional extension with potential function either of the form $f = ar + b$ or $f(x, y) = ar(x) + v(y)$. Then $(M, g_\varphi)$ must be flat, the soliton must be steady, and $\varphi$ is torsion-free.
\end{theorem}

\begin{proof} To prove this, we make several observations leading to $\tr(S^2) = 0$.
For any two vectors $X, Y \in \mathfrak g$, we can write $X = h_1 + ae_7$ and $Y = h_2 + be_7$. In the case of decomposition $\mathfrak h \oplus_D \mathbb Re_7$, we have \begin{align*} [X, Y] &= [h_1, h_2] + a[h_1, e_7] + b[e_7, h_2] + ab[e_7, e_7] \\ &= [h_1, h_2] -aD(h_1) + bD(h_2) \in \mathfrak h.\end{align*} By similar arguments in the case of decomposition $\mathfrak h' \oplus_{D'}\mathbb R \nabla r$, we also get $[X, Y] \in \mathfrak h'$. Thus $[X, Y] \in \mathfrak h \cap \mathfrak h'$ $\forall$ $X, Y \in \mathfrak g.$ This shows that $D: \mathfrak h \rightarrow \mathfrak h \cap \mathfrak h' \subset \mathfrak h$ and $D':\mathfrak h' \rightarrow \mathfrak h \cap \mathfrak h'$. In particular, $[e_7, \nabla r], [\nabla r, e_7] \in \mathfrak h \cap \mathfrak h'.$
    
Note that \begin{align*}\hess f(e_7, e_7) = g(\nabla_{e_7}\nabla f, e_7) &= ag(\nabla_{e_7}\nabla r , e_7) \\ &= ag([e_7, \nabla r] + \nabla_{\nabla r}e_7, e_7) =^* ag(\nabla_{\nabla r} e_7, e_7 ) = \frac{a}{2}D_{\nabla r}\|e_7\|^2 = 0,  \end{align*} where equality $*$ follows from the fact that $[e_7, \nabla r] \perp e_7$ as $[e_7, \nabla r] \in \mathfrak h \cap \mathfrak h' \subset \mathfrak h.$ So $\hess f(e_7, e_7) = 0$. Then from the soliton equation, $$0 = \hess f (e_7, e_7) = -\ric_D(e_7, e_7) - \frac{1}{2}\tau_D^2(e_7, e_7) + \frac{1}{3}(\scal_D - \lambda).$$ Since $\tau_D^2(e_7) = 0$ by Proposition 4.5 (1), the preceding soliton equation holds if and only if \begin{equation}\label{eq:4.2} \frac{1}{3}(\scal_D - \lambda) = \ric_D(e_7, e_7) = - \tr(S^2).\end{equation} Moreover, since $-\tr(S^2) = -\frac{1}{4}\tr(D + D^t)^2 = \scal_D$, we get $$\lambda = -2\scal_D.$$ 

We claim $e_7 \in \ker \hess r$. For $i \ne 7$, we also have from the soliton equation that \begin{align*} a\hess r(e_7, e_i) = \hess f(e_7, e_i) &= -g(\ric_D(e_7), e_i) - \frac{1}{2}g(\tau_D^2(e_7), e_i) + g(\frac{1}{3}(\scal_D - \lambda)e_7, e_i) \\ &= g(-(-\tr(S^2)e_7) - \tr(S^2)e_7, e_i) = 0.\end{align*} Thus $e_7 \in \ker \hess r$ and we have $\Span\{e_7, \nabla r\} \subset \ker \hess r$.

Recall $\mathfrak h = \Span\{e_1, ... , e_6\}$. Consider $ \mathfrak h \cap \Span\{\nabla r, e_7\}$, which is at least a one-dimensional subspace containing $\nabla r$, and suppose $\eta$ is in this intersection. Then we can write $\eta = \alpha \nabla r + \beta e_7$ for some $\alpha, \beta \in \mathbb R$. If $\eta = 0$, since $\nabla r$ and $e_7$ are both of unit length, we would get $\nabla r = \pm e_7$, contradicting our assumption. So $\eta \ne 0$.

We claim that both $D(\eta), D^t(\eta)$ are $0$. To show this, we first show $D^t(\eta) = 0$. We then show $S(\eta) = 0$, from which we get $D(\eta) = 0$. Since $\eta \in \mathfrak h$ and $D: \mathfrak h \rightarrow \mathfrak h \cap \mathfrak h'$, it follows that $D(\eta) \in \mathfrak h \cap \mathfrak h'$. By assumption, $e_7 \perp \mathfrak h$ and so $e_7 \perp \mathfrak h \cap \mathfrak h'$. Similarly, $\nabla r \perp \mathfrak h \cap \mathfrak h'$. Hence $D(\eta) \perp \eta$. More generally, $D(v) \perp \eta$ for any $v \in \mathfrak g$. This means $0 = g(D(v), \eta) = g(v, D^t(\eta))$ for any $ v \in \mathfrak g$. Thus $D^t(\eta) = 0$.

To show $S(\eta) = 0$, we need to following. \begin{enumerate}
    \item  $\nabla_{e_7}e_7 = 0$ since by Koszul formula $g(\nabla_{e_7}e_7, e_j) = g([e_j, e_7], e_7) = 0$ for all $j = 1, ..., 6$ as $[e_j, e_7] \in \mathfrak h$ for all $j = 1, ..., 6$; clearly $g(\nabla_{e_7}e_7, e_7) = 0$. 
    \item  We show $\nabla_{\nabla r}e_7 = 0$. Let $X$ be any invariant vector field. Then \begin{align*}
        g(\nabla_{\nabla r}e_7, X) &= \nabla_{\nabla r}(\underbrace{g(e_7, X)}_{\text{constant}}) - g(e_7, \nabla_{\nabla r}X) \\ &= -g(e_7,\nabla_{\nabla r}X) \\ &=  -[g(e_7, \underbrace{[\nabla r, X]}_{\in \mathfrak h \cap \mathfrak h'}) + g(e_7, \nabla_X\nabla r)] \\ &= - g(\nabla_X\nabla r, e_7) = -g(\nabla_{e_7}\nabla r, X) = 0,
    \end{align*} where the last equality follows from $e_7 \in \ker\hess r$.
\end{enumerate}Now recall that the shape operator corresponding to decomposition $\mathfrak h \oplus_D \mathbb Re_7$ is $$T(X) = \nabla_{X}e_7 = -S(X) = -2^{-1}(D + D^t)(X).$$ Then $$-S(\eta) = T(\eta) = \nabla_\eta e_7 = \nabla_{\alpha \nabla r + \beta e_7}e_7 = \alpha \nabla_{\nabla r}e_7 + \beta \nabla_{e_7}e_7 = 0.$$ Thus $S(\eta) = 0$. This together with $D^t(\eta) = 0$ yields $D(\eta) = 0$.

The soliton equation applied to $\eta$ is

$$\hess f(\eta, \eta) = -\ric_D(\eta, \eta) - \frac{1}{2}\tau_D^2(\eta, \eta) + \frac{1}{3}(\scal_D - \lambda)g(\eta, \eta).$$ Since $\eta \in \mathfrak h$, the operators $\ric_D$ and $\tau_D^2$ applied $\eta$ is equal to the restriction to their $6 \times 6$ diagonal blocks applied to $\eta$. These blocks only involve $D$ and $D^t$ and since $D(\eta) = D^t(\eta) = 0$, these operators applied to $\eta$ are $0$. So $\ric_D(\eta, \eta), \tau_D^2(\eta, \eta)$ are both $0$. As $\eta \in \Span\{\nabla r, e_7\} \subset \ker \hess r$,  $\hess f(\eta, \eta) = 0$. We get $$0 = \frac{1}{3}(\scal_D - \lambda)g(\eta, \eta).$$ Since $\eta \ne 0$, $g(\eta, \eta) = \|\eta\|^2 > 0$. So for the equality to hold, we must have $\frac{1}{3}(\scal_D - \lambda) = 0$, from which it follows that $\tr(S^2) = 0$ by \ref{eq:4.2}. We can now apply the same arguments as in the proof of Theorem 4.10 to conclude that the space is flat, the soliton is steady, and that $\varphi$ is torsion-free. 
\end{proof}

\noindent \textit{Acknowledgements.} I want to thank my PhD advisor Prof. William Wylie for his patience, guidance, and support in the process of my writing this paper.

\bibliographystyle{amsplain}

\end{document}